\documentclass[11pt]{amsart}

\usepackage{amsfonts}
\usepackage{amssymb}
\usepackage{hyperref}
\usepackage{stmaryrd}

\usepackage{tikz}
\usetikzlibrary{matrix,arrows,decorations.pathmorphing}

 \usepackage{hyperref}
\usepackage[capitalise]{cleveref}

\theoremstyle{plain}
\newtheorem{thm}{Theorem}[section]
\newtheorem{lem}[thm]{Lemma}
\newtheorem{cor}[thm]{Corollary}
\newtheorem{prop}[thm]{Proposition}

\theoremstyle{definition}

\theoremstyle{remark}
\newtheorem{remark}[thm]{Remark}

\def \F {\mathbb{F}}

\def \Z {\mathbb{Z}}

\def \lcm {\mathrm{lcm}}
\def \tr {\operatorname{tr}}
\def \Frob {\operatorname{Frob}}

\makeatletter
\newcommand*{\defeq}{\mathrel{\rlap{%
                     \raisebox{0.27ex}{$\m@th\cdot$}}%
                     \raisebox{-0.27ex}{$\m@th\cdot$}}%
                     =}

\newcommand*{\eqdef}{=\mathrel{\rlap{%
                     \raisebox{0.27ex}{$\m@th\cdot$}}%
                     \raisebox{-0.27ex}{$\m@th\cdot$}}%
                     }

\numberwithin{equation}{section}
\makeatother

\makeatletter
\def\@setcopyright{}
\def\serieslogo@{}
\makeatother

\input xy
\xyoption{all} 

\title{On the Chowla and twin primes conjectures over $\mathbb F_q[T]$}

\author{Will Sawin}\thanks{W.S. served as a Clay Research Fellow while working on this paper.}

\author{Mark Shusterman}

\address{Department of Mathematics, Columbia University, New York, NY 10027, USA}
\email{sawin@math.columbia.edu}

\address{Department of Mathematics, University of Wisconsin-Madison, 480 Lincoln Drive, Madison, WI 53706, USA}
\email{mshusterman@wisc.edu}

\begin{document}

\begin{abstract}

Using geometric methods, we improve on the function field version of the Burgess bound,
and show that, when restricted to certain special subspaces, the M\"{o}bius function over $\F_q[T]$ can be mimicked by Dirichlet characters.
Combining these, we obtain a level of distribution close to $1$ for the M\"{o}bius function in arithmetic progressions,
and resolve Chowla's $k$-point correlation conjecture with large uniformity in the shifts.
Using a function field variant of a result by Fouvry-Michel on exponential sums involving the M\"{o}bius function,
we obtain a level of distribution beyond $1/2$ for irreducible polynomials,
and establish the twin prime conjecture in a quantitative form. All these results hold for finite fields satisfying a simple condition.


\end{abstract}

\maketitle

\section{Introduction}

Our main results are the resolutions of two open problems in number theory, except with the ring of integers $\Z$ replaced by the ring of polynomials $\F_q[T]$, under suitable assumptions on $q$. 

We first fix some notation. Define the norm of a nonzero $f \in \F_q[T]$  to be
\begin{equation}
|f| = q^{\deg(f)} = |\F_q[T]/(f)|.
\end{equation}
The degree of the zero polynomial is negative $\infty$, so we set its norm to be $0$. 

\subsection{The main result - twin primes}

Our main result covers the twin prime conjecture in its quantitative form. The latter is the $2$-point prime tuple conjecture of Hardy-Littlewood, predicting for a nonzero integer $h$ that
\begin{equation}
\#\{X \leq n \leq 2X : n \ \text{and} \ n+h \ \text{are prime}\}  \sim \mathfrak{S}(h)\frac{X}{\log^2(X)}, \quad X \to \infty,
\end{equation}
where 
\begin{equation}
\mathfrak{S}(h) = \prod_p (1 - p^{-1})^{-2}(1 - p^{-1} - p^{-1} \mathbf 1_{p \nmid h}),
\end{equation}
with $\mathbf 1_{p \nmid h}$ equals $1$ if $h$ is not divisible by $p$, and $0$ otherwise.

For the function field analogue, we set
\begin{equation}
\mathfrak{S}_q(h) = \prod_P \left(1 - |P|^{-1} \right)^{-2} \left( 1 - |P|^{-1} - |P|^{-1} \mathbf 1_{P \nmid h} \right)
\end{equation}
where $q$ is a prime power, $P$ ranges over all primes (monic irreducibles) of $\F_q[T]$, and $h \in \F_q[T]$ is nonzero.

\begin{thm}\label{TwinPrimesRes}

For an odd prime number $p$, and a power $q$ of $p$ satisfying $q > 685090p^2$, the following holds.
For any nonzero $h \in \F_q[T]$ we have
\begin{equation}
\#\{f \in \F_q[T]: |f| = X, f \ \text{and} \ f+h \ \text{are prime}\} \sim \mathfrak{S}_q(h)\frac{X}{\log_q^2(X)}
\end{equation}
as $X \to \infty$ through powers of $q$.
Moreover, we have a power saving (depending on $q$) in the asymptotic above.

\end{thm}

For example, the $2$-point Hardy-Littlewood conjecture holds over
\begin{equation}
\F_{3^{15}}, \F_{5^{11}}, \F_{7^9}, \F_{11^8}, \F_{685093^3}.
\end{equation}

In case $h \in \F_q[T]$ is a monomial, 
the fact that the count above tends to $\infty$ as $X \to \infty$ (under the weaker assumption $q>105$) has been proved in \cite[Theorem 1.3]{CHLPT15} using an idea of Entin. Their work builds on the recent dramatic progress on this problem over the integers, particularly \cite{Ma15}. The strongest result known over the integers is \cite[Theorem 16(i)]{PM14}, which says that for any `admissible tuple' of 50 integers, there exists at least one difference $h$ between two elements in the tuple such that there are infinitely many pairs of primes separated by $h$. 

\begin{remark}

Our proof of \cref{TwinPrimesRes} establishes also the analog of the Goldbach problem over function fields, and can be modified to treat more general linear forms in the primes.

\end{remark}

The proof of \cref{TwinPrimesRes} passes through some intermediate results which may be of independent interest. We will discuss these results in the remainder of the introduction. Once \cref{ThirdRes} and \cref{FourthRes} below are established, \cref{TwinPrimesRes} will follow from a quick argument involving a convolution identity relating the von Mangoldt function, which can be used to count primes, to the M\"{o}bius function.

\subsection{ The key ingredient - Chowla's conjecture}

The main ingredient in the proof of \cref{TwinPrimesRes} is the removal of the `parity barrier'.
More precisely, we confirm Chowla's $k$-point correlation conjecture over $\F_q[T]$ for some prime powers $q$. 
Over the integers, this conjecture predicts that
for any fixed distinct integers $h_1, \dots, h_k$, one has
\begin{equation}
\sum_{n \leq X} \mu(n+h_1) \mu(n+ h_2) \cdots \mu(n + h_k) = o(X), \quad X \to \infty.
\end{equation}
The only completely resolved case is $k=1$ which is essentially equivalent to the prime number theorem.

For the function field analogue, we recall that the M\"{o}bius function of a monic polynomial $f$ is given by
\begin{equation}
\mu(f) =
\begin{cases}
0, \quad \# \{P : P^2 \mid f\} > 0 \\
1, \quad \# \{P : P \mid f\} \equiv 0 \ \mathrm{mod} \ 2 \\
-1, \ \# \{P : P \mid f\} \equiv 1 \ \mathrm{mod} \ 2,
\end{cases}
\end{equation}
and denote by $\F_q[T]^{+}$ the set of monic polynomials over $\F_q$.

\begin{thm} \label{SecRes}

For an odd prime number $p$, an integer $k \geq 1$, and a power $q$ of $p$ satisfying $q > p^2 k^2 e^2 $, the following holds. 
For distinct $h_1, \dots, h_k \in \F_q[T]$ we have
\begin{equation}\label{SecResEq}
\sum_{\substack{ f \in \F_q[T]^{+} \\ |f| \leq X}} \mu(f + h_1) \mu(f + h_2) \cdots \mu(f + h_k) = o(X), \quad X \to \infty.
\end{equation}
\end{thm}

For instance, the $2$-point Chowla conjecture holds over
\begin{equation}
\mathbb F_{3^6}, \mathbb F_{5^5}, \mathbb F_{7^4} , \mathbb F_{31^3}.
\end{equation}

In fact, in \cref{SecRes} we obtain a power saving inversely proportional to $p$,
and the shifts $h_1, \dots, h_k$ can be as large as any fixed power of $X$ (the corresponding assumption on $q$ becomes stronger as this power grows larger). 
In \cref{MobPrimCor} we also get cancellation in case the sum is restricted to prime polynomials $f$.

Over the integers, the $k=2$ case of the Chowla conjecture, with logarithmic averaging, was proven by Tao \cite[Theorem 3]{Tao16}, building on earlier breakthrough work of Matom\"{a}ki and Radziwi\l\l~\cite{MR16}. The $k$ odd case, again with logarithmic averaging, was handled by Tao and Ter\"{a}v\"{a}inen \cite{TT17}. Generalizations of some of these arguments to the function field setting are part of a work in progress by Klurman, Mangerel, and Ter\"{a}v\"{a}inen.

In contrast to these works, which deal with any sufficiently general (i.e. non-pretentious) multiplicative function, our result relies on special properties of the M\"{o}bius function (in positive characteristic). Specifically, we observe that for any fixed polynomial $r$, the function $\mu(r + s^p)$ essentially equals $\chi_{D_r} (s+ c_r)$ where $\chi_{D_r}$ is a quadratic Dirichlet character and $c_r$ is a shift, both depending only on $r$ (and not on $s$). 
This observation is very closely related to the properties of M\"{o}bius described in \cite{CCG}, specifically \cite[Theorem 4.8]{CCG}. Conrad, Conrad, and Gross prove a certain quasiperiodicity property in $s$ for a general class of expressions of the form $r+s^p$, while we give a more precise description via Dirichlet characters in a special case. 

Our observation on $\mu(r+s^p)$ arises from the connection between the parity of the number of prime factors of a squarefree $f \in \F_q[T]$, and the sign (inside the symmetric group) of the Frobenius automorphism acting on the roots of $f$.
In odd characteristic, 
this sign is determined by the value of the quadratic character of $\F_q^\times$ on the discriminant of $f$, i.e. the resultant of $f$ and its derivative $f'$. In characteristic $p$, the derivative of $f = r+ s^p$ is equal to the derivative of $r$, so the aforementioned sign of Frobenius is determined by the quadratic character of the resultant of $f$ with the fixed polynomial $r'$.
The latter is a quadratic Dirichlet character of $f$, and thus equals an additively shifted Dirichlet character of $s$.

To use our observation, we restrict the sum in \cref{SecRes} to $f$ of the form $r+s^p$ for any fixed $r$, and obtain a short sum in $s$ of a product of additively shifted Dirichlet characters. 
As the conductors of these characters are typically essentially coprime, 
we arrive at a short sum of a single Dirichlet character.

Typically in analytic number theory, short character sums are handled by the method of Burgess, who showed in \cite{Bur63} that for a real number $\eta > 1/4$, a squarefree integer $M$, 
a real number $X \geq |M|^{\eta}$, 
and a nonprincipal Dirichlet character $\chi$ mod $M$, one has 
\begin{equation}
\sup_{s \in \Z} \left| \sum_{|a| \leq X} \chi(s + a) \right| = o(X), \quad \ |M| \to \infty.
\end{equation}
Refining the method of Burgess is the focus of several works, 
but the exponent $1/4$ has not yet been improved (even conditionally).  However, in the function field setting, we can do better by a geometric method, as long as $q$ is sufficiently large.

\begin{thm} \label{FirstRes}

Fix $\eta > 0$.
Then for a prime power $q > e^2/\eta^2  $  the following holds.
For a squarefree $M \in \F_q[T]$,
a real number $X \geq |M|^\eta$, 
and a nonprincipal Dirichlet character $\chi$ mod $M$,
we have
\begin{equation}
\sup_{s \in \F_q[T]} \left| \sum_{|a| \leq X} \chi(s + a) \right| = o(X), \quad \ |M| \to \infty.
\end{equation}

\end{thm}

By further enlarging $q$, we get arbitrarily close to square root cancellation. 
This is stated more precisely in \cref{FFCSCor}.

To prove \cref{FirstRes},
we express the problem geometrically, viewing the short interval $\{s+a : |a| \leq X \}$ as an affine space over $\mathbb F_q$, and the character $\chi$ as arising from a sheaf on that space. Following a strategy from \cite[appendix by Katz]{Hoo91}, we use vanishing cycles theory to compare the cohomology of this sheaf for the $s = 0$ short interval and its cohomology for a general short interval. 
Vanishing cycles can only occur when the vanishing locus of $\chi$ is not (geometrically) a simple normal crossings divisor. 
Arguing as in \cite{Kat89}, we split the modulus $M$ of $\chi$ into a product of distinct linear terms over $\overline{\F_q}$, 
which makes our vanishing locus a union of the hyperplanes where the linear terms vanish, so we can check that this is a simple normal crossings divisor away from some isolated points. This implies that the cohomology groups vanish until almost the middle degree.
Since we can precisely calculate the vanishing cycles at the isolated points, we get a very good control of the dimensions of cohomology groups as well.

\begin{remark}
The relation between M\"{o}bius and multiplicative characters is less powerful the larger $p$ is,
as then fewer polynomials share a given derivative.
On the other hand, our geometric character sum bounds become stronger as $q$ grows. 
Thus, to make this method of proving Chowla work, we need $q$ to be sufficiently large with respect to $p$.
\end{remark}

\begin{remark}
The  study of the statistics of polynomial factorizations by examining Frobenius as an element of the symmetric group has been very fruitful in the `large finite field limit',
where (in the notation of \cref{SecRes}) $X$ is kept fixed and $q$ is allowed to grow.
We refer to \cite{CaRu14}, \cite{Ca15}, \cite{GS18}  (and references therein) for the large finite field analogs of \cref{SecRes} which save a fixed power of $q$.
Our methods likely give an improved savings in the large finite field limit when the characteristic is fixed, as long as the degrees of the polynomials are sufficiently large with respect to the characteristic, but we have not carefully calculated the resulting bounds in this range.
\end{remark}


\subsection{Further ingredients - level of distribution estimates} 

Another ingredient in the proof of \cref{TwinPrimesRes} is an improvement of the level of distribution of the M\"{o}bius and von Mangoldt functions in arithmetic progressions.
Over the integers, 
assuming the Generalized Riemann Hypothesis (GRH), 
this level of distribution is (at least) $1/2$, which means that
\begin{equation}
\sum_{\substack{n \leq X \\ n \equiv a \ \mathrm{mod} \ M}} \mu(n) = o \left(\frac{X}{|M|}\right), 
\quad \sum_{\substack{n \leq X \\ n \equiv a \ \mathrm{mod} \ M}} \Lambda(n) = \frac{X}{\varphi(M)} + o \left(\frac{X}{|M|}\right)
\end{equation}
where $M, a$ are coprime integers, and $|M| \leq X^{\frac{1}{2} - \epsilon}$ (for any fixed $\epsilon > 0$ and $A \in \mathbb R$).
For M\"{o}bius over $\F_q[T]$, we obtain a level of distribution close to $1$.

\begin{thm} \label{ThirdRes}

Fix $\eta > 0$.
For an odd prime number $p$, and a power $q$ of $p$ with $q >p^2 e^2 \left(\frac{2}{\eta} -1\right)^2  $, the following holds.
For coprime $M,a \in \F_q[T]$, and a real number $X$ with $X^{1-\eta} \geq |M|$ we have
\begin{equation}
\sum_{\substack{f \in \F_q[T]^{+} \\ |f| \leq X \\ f \equiv a \ \mathrm{mod} \ M}} \mu(f) = o\left(\frac{X}{|M|}\right), 
\quad |M| \to \infty.
\end{equation}

\end{thm}

As in the previous theorems, we obtain a power savings estimate.
Here however (as opposed to \cref{SecRes}), every $f$ in our sum may have a different derivative,
so a somewhat more elaborate implementation of our observation on the M\"{o}bius function is required.
We put $f =Mg + a$, and wishfully write
\begin{equation}
\mu(Mg + a) \approx \mu(M) \mu \left( g + \frac{a}{M} \right)
\end{equation}
in order to create coincidences among the derivatives of the inputs to the M\"{o}bius function.
This is carried out more formally in \cref{PreparingMobiusToArithProgLem},
where we show that for a power $q$ of an odd prime $p$, 
and coprime $a, M \in \mathbb F_q[T]$, the function $s \mapsto \mu( a + s^p M)$ is essentially proportional to an additive shift of a (quadratic) Dirichlet characters in $s$, with the modulus of the character depending on $a$ and $M$ in an explicit way. 
To visualize the power of this claim, we view $\mathbb F_q[T]$ as a rank $p$ lattice over its subring $\mathbb F_q[T^p]$. Restricting the M\"{o}bius function to any line in this lattice gives a Dirichlet character whose modulus varies with the line.

%
%
%
%
%

\vspace{10pt}

In order to deduce from \cref{ThirdRes} an improved level of distributions for primes,
we establish in the appendix a function field variant of \cite{FM98} giving quasi-orthogonality of the M\"{o}bius function and `inverse additive characters'.
While Fouvry-Michel work with characters to prime moduli, in order to establish \cref{TwinPrimesRes} we need arbitrary squarefree moduli.

\begin{thm} \label{French}

Let $q$ be a prime power, and let $\epsilon>0$.
Then for a squarefree $M \in \F_q[T]$, and an additive character $\psi$ mod $M$, we have
\begin{equation} \label{FMeq}
\sum_{\substack{f \in \F_q[T]^+ \\ |f| \leq X \\ (f, M) = 1}} \mu(f)\psi \left( \overline f \right) \ll |M|^{ \frac{3}{16} + \epsilon} X^{\frac{25}{32}}, \quad X,|M| \to \infty
\end{equation}
where $\overline f$ denotes the inverse of $f$ mod $M$, and the implied constant depends only on $q$ and $\epsilon$.

\end{thm}

For nonzero $M \in \F_q[T]$ we define Euler's totient function by
\begin{equation}
\varphi(M) = \left| \left( \F_q[T]/(M) \right)^\times \right|,
\end{equation}
and for $f \in \F_q[T]^+$, we define the von Mangoldt function by 
\begin{equation}
\Lambda(f) = 
\begin{cases}
\deg(P), \quad f = P^n \\
0, \quad \text{otherwise}.
\end{cases}
\end{equation}

\begin{thm} \label{FourthRes}

Fix $\delta < \frac{1}{126}$. 
For an odd prime $p$ and a power $q$ of $p$ with 
\begin{equation}
q > p^2e^2\left( \frac{51 - 26\delta}{1 - 126\delta} \right)^2,
\end{equation}
the following holds.
For a squarefree $M \in \F_q[T]$, a polynomial $a \in \F_q[T]$ coprime to $M$,
and $X$ a power of $q$ with $X^{\frac{1}{2} + \delta} \geq |M|$ we have
\begin{equation}
\sum_{\substack{f \in \F_q[T]^{+} \\ |f| \leq X \\ f \equiv a \ \mathrm{mod} \ M}} \Lambda(f) = \frac{X}{\varphi(M)}
+ o\left( \frac{X}{|M|} \right),
\quad |M| \to \infty.
\end{equation}

\end{thm}

We have not ventured too much into improving the constant $\frac{1}{126}$,
as our method cannot give anything above $\frac{1}{6}$,
even if \cref{French} would give square root cancellation.
Since our proof of \cref{FourthRes} is based on the `convolutional' connection of von Mangoldt and M\"{o}bius,
it is not surprising that a level of distribution of $\frac{2}{3} = \frac{1}{2} + \frac{1}{6}$, 
which is a longstanding barrier for the divisor function over $\Z$ (perhaps the most basic convolution), 
is a natural limit of our techniques. A large finite field variant of \cref{ThirdRes} and \cref{FourthRes} was earlier proved in \cite[Theorem 2.5]{BBSR}.



\begin{remark}
It would be interesting to see whether our results can be extended to characteristic $2$,
perhaps in a manner similar to which \cite{Ca15} extends the results of \cite{CaRu14}.
\end{remark} 
 

\subsection{Additional results in small characteristic}

Throughout this work, 
we have not made every possible effort to reduce the least values of the prime powers $q$ to which our theorems apply.
Instead, we present in the last section some results that hold for $q$ as small as $3$.

The first concerns sign changes of M\"{o}bius in short intervals.
Improving on many previous works,
Mat\"{o}maki and Radziwi\l{}\l{} have shown in \cite{MR16} that for any $\eta > 1/2$,
and any large enough positive integer $N$,
there exist integers $a,b$ with $|a|, |b| \leq N^\eta$ such that $\mu(N + a) = 1, \ \mu(N + b)=-1$.
In characteristic $3$, we show that the exponent $1/2$ can be improved to $3/7$.

\begin{thm} \label{SmallqRes1}

Let $q$ be a power of $3$, and fix $3/7 < \eta < 1$.
Then for any $f \in \F_q[T]^+$ of large enough norm,
there exist $g,h \in \F_q[T]$ with $|g|,|h| \leq |f|^\eta$ such that $\mu(f+g) = 1$ and $\mu(f+h) = -1$.

\end{thm}

We follow the same proof strategy relating M\"{o}bius to characters,
but since we allow small values of $q$, 
we cannot apply \cref{FirstRes} anymore.
Now however, once we are interested in sign change only (and not cancellation), 
we can focus on just one of the derivatives appearing. 
It turns out that if this derivative has a relatively large order of vanishing at $0$,
the conductor of the associated character is relatively small, 
and we can apply a function field version (see \cite{Hsu99}) of the aforementioned result of Burgess.
For $p > 3$, the character sums arising are too short for the Burgess bound to apply.

For large enough $q$, 
\cref{SmallqRes1} follows from \cref{ThirdRes} (since $T \mapsto 1/T$ allows one to think of short intervals as arithmetic progressions),
and also from the $k=1$ case of \cref{SecRes} (as we have sufficient uniformity in the shift).

Our last result, in the spirit of \cite{Gal72}, shows that M\"{o}bius enjoys cancellation in the arithmetic progression $1$ mod a growing power of a fixed prime $P$, no matter how slowly does the length of the progression increase.

\begin{thm} \label{SmallqRes2}

Let $q$ be a power of $3$.
Fix an irreducible $P \in \F_q[T]$. 
Then for a positive integer $n$ we have
\begin{equation}
\sum_{\substack{f \in \F_q[T]^+ \\ |f| \leq X \\ f \equiv 1 \ \mathrm{mod} \ P^n}} \mu(f) = o \left( \frac{X}{|P|^n}\right),
\quad \frac{X}{|P|^n} \to \infty.
\end{equation}

\end{thm}

As before, we can obtain a power saving.
Since the progressions are so short, this result does not follow from the previous ones, even if $q$ is large.
In view of Maier's phenomenon, we cannot expect to obtain the theorem for all residue classes.

\subsection{Further directions}

In future work, we hope to use some of the methods introduced here to address the following problems:

\begin{itemize}

\item Obtaining cancellation in `polynomial M\"{o}bius sums' such as
\begin{equation*}
\sum_{\substack{ f \in \F_q[T]^+ \\ |f| \leq X}} \mu \left( f^2 + T \right)
\end{equation*}
which may be relevant for counting primes of the form $f^2 + T$.

\item Calculating the variance (and higher moments) of the M\"{o}bius function in short intervals (and arithmetic progressions) over $\F_q[T]$. 




\end{itemize}

\subsection{Notation}

From this point on, it will be more convenient to work with degrees of polynomials instead of absolute values.
For $g \in \F_q[T]$ we denote its degree by $d(g)$. By convention, the latter is $-\infty$ if $g=0$.
The letter $q$ denotes a prime power, and is often suppressed from notation such as
\begin{equation}
\mathcal{M}_d = \left \{g \in \F_q[T]^+ : d(g) = d \right \}.
\end{equation}

\section{Character sums}

The main result of this section is the following.

\begin{thm}\label{main-character-sum}

Let $t \leq m$ be natural numbers, let $f \in \F_q[T]$, let $g \in \mathcal{M}_m$ be squarefree, and let $\chi: \left(\mathbb F_q[T]/g\right)^\times \to \mathbb C^\times$ be a nontrivial character. Then

\begin{equation}\label{main-character-sum-equation}
\left| \sum_{\substack{ h \in \mathbb F_q[T] \\ d( h) < t \\ \gcd(f+h,g)=1 }} \chi(f+h) \right| \leq 
( q^{1/2} +1) {m-1 \choose t} q^{ \frac{t}{2} }.
\end{equation}
\end{thm}

To prove this theorem, we use the following geometric setup. View $\mathbb A^t$ as a space parameterizing polynomials $h$ of degree less than $t$ and let $(c_1,c_2) $ be coordinates on $\mathbb A^2$. Let $U \subseteq \mathbb A^t \times \mathbb A^2$ be the open set consisting of points $(h,(c_1,c_2))$ where $c_1 f+  h + c_2 T^t$ is prime to $g$. Let $j: U \to \mathbb P^t \times \mathbb A^2$ be the open immersion, embedding $\mathbb A^t$ into $\mathbb P^t$ in the usual way. Let $\pi: \mathbb P^t \times \mathbb A^2 \to \mathbb A^2$ be the projection.

Let $\mathcal T$ be the torus parameterizing of polynomials of degree less than $m$ that are relatively prime to $g$. (The space $\mathcal T$ isa torus because, over an algebraically closed field, we may factor $g$ and write $\mathcal T$ as $\mathbb G_m^m$, with the coordinates given by evaluating the polynomial on each of the roots of $g$.)
On $\mathcal T$, we have a character sheaf $\mathcal L_\chi$ whose trace function is $\chi$, constructed by the Lang isogeny.  Let $\mathcal L_\chi ( c_1 f +  h + c_2 T^t)$ be the pullback of this sheaf to $U$ along the natural map from $U$ to $\mathcal T$ that sends $(h,(c_1,c_2) )$ to $c_1 f+ h +c_2  T^t$. We will prove \eqref{main-character-sum-equation} using geometric properties of the complex $R \pi_* j_! \mathcal L_\chi ( c_1 f +  h + c_2 T^t)$.

\begin{lem}\label{scale-invariance} The complex $R \pi_* j_! \mathcal L_\chi ( c_1 f +  h + c_2 T^t)$ is geometrically isomorphic to its pullback under the map $(c_1,c_2) \to (\lambda c_1, \lambda c_2)$ for any $\lambda \in \overline{\F_q}^\times$. \end{lem}

\begin{proof} Because $U$ is invariant under multiplying  $c_1,c_2,$ and $h$ by $\lambda$, as are the maps $j$ and $\pi$, its suffices to check that $\mathcal L_\chi ( c_1 f +  h + c_2 T^t)$ is geometrically isomorphic to its pullback under this multiplication map. It is sufficient that $\mathcal L_\chi$ on $\mathcal T$ is geometrically isomorphic to its pullback under any multiplicative translation, which follows from its construction as a character sheaf. \end{proof}

\begin{lem}\label{l-function-case} The stalk of $R \pi_* j_! \mathcal L_\chi ( c_1 f +  h + c_2 T^t)$ at the point $(0,1)$ is supported in degree $t$, where it has rank ${m-1 \choose t}$. \end{lem}

\begin{proof} 

When $c_1=0, c_2=1$, the polynomial $c_1 f+ h+ c_2 T^t = T^t +h$ is monic and has degree $t$. By the proper base change theorem, 
our stalk is thus equivalent to the compactly supported cohomology of the space of degree $t$ monic polynomials that are prime to $g$, with coefficients in $\mathcal L_\chi(T^t + h)$. 
We may view this space as the quotient $\left(\operatorname {Spec} \mathbb F_q [x ,g(x) ^{-1} ] \right)^t /S_t$, and denote by 
\begin{equation}
\rho \colon \left(\operatorname {Spec} \mathbb F_q \left[x ,g(x) ^{-1} \right] \right)^t \to \left(\operatorname {Spec} \mathbb F_q \left[x ,g(x) ^{-1} \right] \right)^t/S_t
\end{equation}
the quotient map. 
Then $(\rho_* \rho^* \mathbb Q_\ell)^{S_t}= \mathbb Q_\ell$, so 
\begin{equation}
\left(\rho_* \rho^* \mathcal L_\chi ( T^t + h)\right)^{S_t} = \mathcal L_\chi(T^t+h)
\end{equation}
and thus
\begin{equation}
\begin{split}
&H^* \left( \left(\operatorname {Spec} \mathbb F_q \left[x ,g(x) ^{-1} \right] \right)^t/S_t ,  \mathcal L_\chi(T^t+h) \right) =  \\
&H^* \left( \left(\operatorname {Spec} \mathbb F_q \left[x ,g(x) ^{-1} \right] \right)^t/S_t ,\rho_* \rho^*   \mathcal L_\chi(T^t+h)\right)^{S_t} = \\
&H^* \left( \left(\operatorname {Spec} \mathbb F_q \left[x ,g(x) ^{-1} \right] \right)^t , \rho^*   \mathcal L_\chi(T^t+h)\right)^{S_t}.
\end{split}
\end{equation}

Now $\rho$ is the map defined by factorizing a polynomial into linear terms, so $\rho^* \mathcal L_\chi( T^t +h) = ( \mathcal L_\chi (T-x) )^{\boxtimes t}$. By the K\"{u}nneth formula, it follows that the stalk of $R \pi_* j_! \mathcal L_\chi ( c_1 f +  h + c_2 T^t)$ at $(0,1)$ is 
\begin{equation}
\left( \left( H^* \left(\operatorname {Spec} \overline{\mathbb F_q}\left[x ,g(x) ^{-1} \right], \mathcal L_\chi (T-x) \right)  \right)^{ \otimes t} \right)^{S_t}.
\end{equation}

Because $\chi$ is nontrivial, $\mathcal L_\chi (T-x)  $ has nontrivial mondromy, so 
\begin{equation}
H^* \left(\operatorname {Spec} \overline{\mathbb F_q} \left[x ,g(x) ^{-1} \right], \mathcal L_\chi (T-x) \right)
\end{equation} 
vanishes in degrees other than $1$. Because it arises from a character sheaf on a torus, $\mathcal L_\chi$ has tame local monodromy, so the rank of this cohomology group in degree $1$ is the Euler characteristic of $\operatorname {Spec} \overline{\mathbb F}_q [x ,g(x) ^{-1} ]$, which is $m-1$. 
Thus 
\begin{equation}
\left( H^* \left(\operatorname {Spec} \overline{\mathbb F_q} \left[x ,g(x) ^{-1} \right], \mathcal L_\chi (T-x) \right)  \right)^{ \otimes t}
\end{equation}
is supported in degree $t$, where it equals the $t$-th tensor power of an $m-1$-dimensional vector space, with $S_t$ acting the usual way, twisted by the sign character because of the Koszul sign in the tensor product of a derived category. Thus taking $S_t$-invariants is equivalent to taking the $t$-th wedge power of this $m-1$-dimensional vector space, which has dimension ${m-1 \choose t}$.\end{proof}

\begin{lem}\label{snc} The complement of $U$ in $\mathbb P^t \times \mathbb A^2$ is a divisor with simple normal crossings relative to $\mathbb A^2$ away from 
\begin{equation}
\left\{ \left(h,(c_1,c_2)\right) \in \mathbb A^t \times \mathbb A^2 :   d \left(\gcd ( c_1f + h+ c_2T^t , g) \right) > t  \right\}.
\end{equation}
\end{lem}

\begin{proof} Because this is a purely geometric question, we may assume that $g$ splits completely, and let $\alpha_1,\dots, \alpha_m$ be its roots. Then the complement of $U$ is the union of the hyperplane at $\infty$ in $\mathbb P^t$ with the hyperplanes 
\begin{equation}
c_1 f(\alpha_i) + h(\alpha_i) + c_2 \alpha_i^t = 0, \quad 1 \leq i \leq m. 
\end{equation}
This union has simple normal crossings if the intersection of each subset of our hyperplanes has the expected dimension.

We first consider the case of a subset not including the hyperplane at $\infty$. 
For any $S \subseteq \{1 ,\dots, m\}$, the intersection of the hyperplanes corresponding to the elements of $S$ is 
\begin{equation}
\left\{ \left(h,(c_1,c_2)\right) \in \mathbb A^t \times \mathbb A^2 : \prod_{i \in S} (T-\alpha_i) \ \Big\rvert \ c_1f + h+ c_2T^t  \right\}.
\end{equation}
Since $\prod_{i \in S} (T-\alpha_i)$ is a polynomial of degree $|S|$, 
the above has codimension $|S|$ for any $c_1, c_2$ as long as $|S| \leq t$, because the set of $h$ of degree less than $t$ which are multiples of this polynomial has the expected dimension. On the other hand, when $|S|>t$ we get that 
\begin{equation}
d \left( \gcd ( c_1f + h+ c_2T^t , g) \right) \geq |S| > t,
\end{equation}
so removing these points, we obtain simple normal crossings. 

Now we consider the case where we have a set of hyperplanes corresponding to $S \subseteq \{1,\dots, m\}$ and also the divisor at $\infty$. We can take coordinates on the divisor at $\infty$ to be $(h, (c_1,c_2))$, with $h$ nonzero and well-defined up to scaling. In these coordinates, the equation for the intersection of any hyperplane with the divisor at $\infty$ is its original equation with all terms having degree zero in $h$ removed. Thus the equation for the $i$-th hyperplane, restricted to the divisor at $\infty$, is simply $h(\alpha_i)=0$, so the intersection of the divisor at $\infty$ with the hyperplanes in $S$ consists of $(h,(c_1,c_2))$ where $h$ is a multiple of $\prod_{i \in S} (T-\alpha_i)$. This can only happen if $|S|<t$, as $|S|$ is the degree of this polynomial and $d( h)<t$, but then our intersection always has the expected dimension, so we have simple normal crossings. \end{proof}

\begin{lem}\label{generic-rank} Away from a finite union of lines through the origin, the complex $R \pi_* j_! \mathcal L_\chi$ is supported in degree $t$ with rank ${ m-1 \choose t}$ . \end{lem}

\begin{proof} Let $l_1: \mathbb A^1 \to \mathbb A^2$ send $c$ to $(c,1)$. We consider the vanishing cycles  at zero $R \Phi_c l_1^* j_! \mathcal L_\chi (f+h+ \lambda T^t)$ of the pullback of $j_! \mathcal L_\chi (f+h+ \lambda T^t)$ under $l_1$. Let us first check that the complement of $U$ is a simple normal crossings divisor everywhere in the fiber over zero. To do this, we apply \cref{snc} and observe that we cannot have $d ( \gcd ( c_1f + h+ c_2T^t , g)) > t$ in this fiber as $c_1=0,c_2=1$ and 
\begin{equation}
d \left( \gcd (  h+ T^t , g) \right) \leq d ( h +  T^t) \leq t.
\end{equation}
By \cite[XIII Lemma 2.1.11]{sga7-ii} it follows that the vanishing cycles 
\begin{equation}
R \Phi_c l_1^* j_! \mathcal L_\chi (f+h+ \lambda T^t)
\end{equation}
vanish everywhere. Hence the cohomology of the nearby and general fibers is isomorphic. Using \cref{l-function-case} to compute the cohomology of the special fiber, it follows that the cohomology of the general fiber is supported in degree $t$ with rank ${m-1 \choose t}$. By constructibility, the stalk cohomology must have the same description at every point in some nonempty open set. By \cref{scale-invariance}, the same description holds for the stalks in the $\mathbb G_m$-orbit of this open set, which is the complement of finitely many lines. \end{proof}

Let $l_2: \mathbb A^1 \to \mathbb A^2$ send $c$ to $(1,c)$.

\begin{lem}\label{vanishing-cycles} 
Taking vanishing cycles at zero, 
the complex 
\begin{equation}
R \Phi_c  l_2^* j_! \mathcal L_\chi ( f + h + \lambda T^t)
\end{equation}
has the following properties:
\begin{itemize}

\item it is supported on $C = \left\{ (h,0) \in \mathbb A^t \times \mathbb A^1 : d \left( \gcd(f+h,g) \right) > t \right\}$;

\item it is supported in degree $t$;

\item the rank of its stalk at $(h,0) \in C$ is ${d (\gcd(f+h,g)) -1 \choose t}$.

\end{itemize}

\end{lem}

\begin{proof}The first property follows immediately from  \cite[XIII Lemma 2.1.11]{sga7-ii} and \cref{snc}.

We note that $l_2^* j_! \mathcal L_\chi ( f + h + \lambda T^t) [t+1]$ is the extension by zero of a lisse sheaf in degree $- (t+1)$ on a variety of dimension $t+1$ and is thus semiperverse.
The dual complex is the pushforward of a lisse sheaf in degree $-(t+1)$ on a variety of dimension $t+1$ along an affine open immersion and thus is semiperverse by Artin's affine theorem.
We conclude that 
\begin{equation}
l_2^* j_! \mathcal L_\chi ( f + h + \lambda T^t) [t+1]
\end{equation}
is a perverse sheaf.
By \cite[Corollary 4.6]{Illusie}, the vanishing cycles of a perverse sheaf are perverse up to a shift by one, so 
\begin{equation}
R \Phi_c  l_2^* j_! \mathcal L_{\chi} (f + h + \lambda T^t)[t]
\end{equation}
is perverse.

The support of the perverse sheaf above is the closed set $C$. Because $C$ does not intersect the divisor at $\infty$, $C$ is finite. 
A perverse sheaf supported at a finite set is necessarily a sum of skyscraper sheaves supported in degree zero, 
so we obtain the second property in the statement of this lemma.

It remains to calculate the rank of the stalk at a particular point $(h_0,0) \in C$.
We do that by working locally in an \'{e}tale neighborhood. 

We set
\begin{equation}
g' = \gcd( f+h_0, g), \quad g^*= g/ g',
\end{equation}
and factor $\mathcal T$ as the product of the torus $\mathcal T'$ of residue classes mod $g'$ and the torus $\mathcal T^*$ of residue classes mod $g^*$. 
This lets us factor $\mathcal L_\chi$ as the tensor product of $\mathcal L_{\chi'}$ (the pullback of a character sheaf from $\mathcal T'$) and $\mathcal L_{\chi^*}$ (the pullback of a character sheaf from $\mathcal T^*$). 
Because $f+h_0$ is relatively prime to $g^*$, the map $U \to \mathcal T \to \mathcal T^*$ extends to a well-defined map in a neighborhood of the point $(h_0,0)$, so $\mathcal L_{\chi^*} ( c_1 f+ h+ c_2 T^t)$ extends to a lisse sheaf in a neighborhood of $(h_0,0)$. 
Because tensoring with a lisse rank one sheaf does not affect vanishing cycles, it suffices to calculate 
\begin{equation}
R \Phi_c  l_2^* j'_! \mathcal L_{\chi'} (c_1 f + h + c_2 T^t), 
\end{equation}
where $j'$ is the inclusion of the open set where $\gcd( c_1 f+ h + c_2 T^t, g')=1$. 

By changing variables, we may replace $(f,h)$ with $f',h'$ where $f' =f+h_0$ and $h'=h-h_0$ - we are then tasked with calculating the vanishing cycles $R \Phi_c  l_2^* j'_! \mathcal L_{\chi'} ( c_1 f' + h' + c_2 T^t)$ at zero. Having done this, we observe that $f'$ is a multiple of $g'$, so translation by $f'$ does not affect $\mathcal L_{\chi'} ( c_1 f' + h' + c_2 T^t)$, thus these vanishing cycles are the same as $R \Phi_c  l_2^* j'_! \mathcal L_{\chi'} ( h' + c_2 T^t)$.

As $d(  h' + c_2 T^t) \leq t$, we can only have 
\begin{equation}
d \left( \gcd( h'+ c_2 T^t, g') \right)>t
\end{equation}
if $h'+ c_2 T^t=0$. 
Hence, by \cref{snc}, the complement of the image of $j'$ is a simple normal crossings divisor way from the point $(0,0)$. 
Therefore, by \cite[XIII Lemma 2.1.11]{sga7-ii}, 
the vanishing cycles are supported at this point. 
By the vanishing cycles long exact sequence, the Euler characteristic of the vanishing cycles complex is the difference between the Euler characteristic of the generic fiber and the Euler characteristic of the special fiber. Because the vanishing cycles sheaf is supported at a single point in a single degree, its Euler characteristic is $(-1)^t$ times its rank at that point. We will calculate these Euler characteristics and thereby calculate the rank.

By \cref{generic-rank}, the Euler characteristic of the generic fiber is $(-1)^t { d( g') - 1 \choose t } $. So it remains to check the Euler characteristic at the special point is zero. Because $\mathcal L_\chi$ is lisse of rank one and tame, the Euler characteristic of the special fiber is the Euler characteristic of the space of polynomials of degree less than $t$ and prime to $g'$. Because this admits a free action of $\mathbb G_m$ by scaling, its Euler characteristic is zero. 
\end{proof}

We can now prove \cref{main-character-sum}.

\begin{proof}

We have a vanishing cycles long exact sequence
\begin{equation*}
\begin{split}
(R \pi_* j_! \mathcal L_\chi ( c_1 f +  h + c_2 T^t))_{(1,0)} &\to (R \pi_* j_! \mathcal L_\chi ( c_1 f +  h + c_2 T^t) )_{(1,\eta)} \\
&\to H^* (\mathbb P^t,  R \Phi_c  l_2^* j_! \mathcal L_\chi ( f + h + \lambda T^t) ).
\end{split}
\end{equation*}

By Lemma \ref{generic-rank}, $(R \pi_* j_! \mathcal L_\chi ( c_1 f +  h + c_2 T^t) )_{(1,\eta)} $ is supported in degree $t$ with rank ${ m-1 \choose t}$. 
By Lemma \ref{vanishing-cycles}, the complex $R \Phi_c  l_2^* j_! \mathcal L_\chi ( f + h + \lambda T^t) $ is supported in degree $t$ and at finitely many points, so the third term above is also supported in degree $t$ and is simply the sum of the stalks at those points, and thus has rank
\begin{equation} \label{RkBoundEq}
r(f, g, t) = \sum_{\substack{ h \in\overline{ \mathbb F_q}[T] \\ d( h )< t \\ d\left( \gcd(f+h,g) \right)>t  }} {d\left( \gcd(f+h,g) \right)-1 \choose t},
\end{equation}
again using Lemma \ref{vanishing-cycles}.

We conclude that, upon suppressing $c_1 f +  h + c_2 T^t$ for brevity, the vanishing cycles long exact sequence becomes
\begin{equation}
\begin{split}
0 &\to (R^t \pi_* j_! \mathcal L_\chi )_{(1,0)} \to  (R^t \pi_* j_! \mathcal L_\chi)_{(1,\eta)} \\
&\to H^t (\mathbb P^t,  R \Phi_c  l_2^* j_! \mathcal L_\chi ) \to ( R^{t+1}  \pi_* j_! \mathcal L_\chi )_{(1,0)}\to 0.
\end{split}
\end{equation} 
Thus $(R \pi_* j_! \mathcal L_\chi ( c_1 f +  h + c_2 T^t))_{(1,0)} $ is supported in degrees $t$ and $t+1$, 
with rank at most ${ m-1 \choose t}$ in degree $t$ and rank bounded by $r(f, g, t)$ in degree $t + 1$. 
By Deligne's theorem, the absolute values of the eigenvalues of Frobenius on the $i$-th cohomology group are at most $q^{i/2}$, so the absolute value of the trace of Frobenius on cohomology is at most 
\begin{equation}
{m-1 \choose t}  q^{\frac{t}{2} } + r(f, g,t) q^{\frac{t+1}{2}}.
\end{equation}

By the Lefschetz fixed point formula, the trace of Frobenius on cohomology equals the sum of the trace of Frobenius on the stalks, which is \begin{equation}
\sum_{\substack{ h \in \mathbb F_q[T] \\ d( h) < t \\ \gcd(f+h,g)=1 }} \chi(f+h).
\end{equation}

Finally, we check that $r(f,g,t) \leq { m-1 \choose t}$.
To do this fix a root $\alpha$ of $g$, 
and note that $ {d( \gcd(f+h,g))-1 \choose t}$ does not exceed the number of degree $t$ divisors of $g$, prime to $T-\alpha$, that divide $f+h$. Each such divisor of $g$ prime to $T-\alpha$ contributes at most once to the sum in \cref{RkBoundEq}, 
so this sum is bounded by the number of such divisors, which is ${ m-1 \choose t} $. 
\end{proof}

\begin{cor} \label{FFCSCor}

Fix $\eta > 0$ and $0< \beta<1/2$.
Then for a prime power $q \geq ( e \eta^{-1})^{\frac{2}{1-2\beta} }  $  the following holds.
For a nonprincipal character $\chi$ to a squarefree modulus $g \in \F_q[T]$, $f \in \F_q[T]$, 
and $t \geq \eta \cdot d(g)$, we have
\begin{equation}
\sum_{d(h) < t} \chi(f + h) \ll q^{ \left(1- \beta \right) t}
\end{equation}
as $t  \to \infty$, with the implied constant depending only on $q$. 

Furthermore, if we have $t \leq \eta \cdot d(g) \leq t'$, then we still have 

\begin{equation}
\sum_{d(h) < t} \chi(f + h) \ll q^{ \left(1- \beta \right) t'}.
\end{equation}

\end{cor}

\begin{proof} If $\eta \geq 1$, the left hand side vanishes and the bound is trivial. 
Otherwise, we apply \cref{main-character-sum} to the left side, taking $m = d(g)$, to obtain
\begin{equation}
\begin{split}
\sum_{d(h) < t} \chi(f + h) &\leq (q^{1/2+1}) {m-1 \choose t} q^{t/2} \\
&\ll {m \choose t} q^{t/2} \leq \left(\frac{t}{m}\right)^{-t} \left(  \frac{m-t}{m} \right)^{t-m}  q^{t/2}
\end{split}
\end{equation}
where the last inequality follows from 
\begin{equation}
\begin{split}
1 = \left( \frac{t}{m} + \frac{m-t}{m} \right)^m &= \sum_{k=0}^m{m \choose k}   \left(\frac{t}{m}\right)^{k} \left(  \frac{m-t}{m} \right)^{m-k} \\ &\geq {m \choose t}   \left(\frac{t}{m}\right)^{t} \left(  \frac{m-t}{m} \right)^{m-t}.
\end{split}
\end{equation}

From the Taylor series we can see that $-\log (1-x) \leq x/(1-x)$ if $x>0$ so $(1-x)^{ - (1-x)/x}  \leq e$. 
Applying this to $x = t/m$, we get 
\begin{equation}
\left(  \frac{m-t}{m} \right)^{t-m} \leq e^t
\end{equation}
so we obtain
\begin{equation}
\sum_{d(h) < t} \chi(f + h)  \ll \left(\frac{t}{m}\right)^{-t}  e^t  q^{t/2}.
\end{equation}

Because $q \geq (e \eta^{-1})^{\frac{2}{1 - 2\beta} }$ and  $t/m \geq \eta$, we have 
\begin{equation}
e  \left(\frac{t}{m}\right)^{-1} \leq e \eta^{-1} \leq q^{ \frac{1}{2} - \beta}
\end{equation}
hence
\begin{equation} \label{InterStepEq}
\left(\frac{t}{m}\right)^{-t}  e^t  q^{t/2} \leq q^{ (1 -\beta) t},
\end{equation}
as desired.

To handle the case where $t \leq \eta \cdot d(g) \leq t'$, first note that we may assume $t' \leq m$. 
We observe that the left hand side of \cref{InterStepEq} is an increasing function of $t$ because its logarithm 
\begin{equation}
t \log m - t\log t + t+ t (\log q)/2
\end{equation} 
has derivative 
\begin{equation}
\log m - \log t + (\log q)/2
\end{equation} 
which is positive in the range $t\leq m$. Thus we can get a bound for the shorter sum which is at least as good as our bound for the longer sum.
\end{proof} 

\section{The M\"{o}bius Function}

From now on, we will assume that the characteristic $p$ of $\F_q$ is odd. Because of this, $\F_q^\times$ admits a unique quadratic character, which we denote $\psi$. We use freely the basic properties of resultants (see \cite{Jan07}) and the Jacobi symbol (see \cite[Chapter 3]{Ros13}).

The following lemma recalls the relation between the (real valued) Jacobi symbol and the quadratic character of a resultant. 

\begin{lem} \label{JacobiLem}
Let $f \in \F_q[T]$, $g = a_nT^t + \dots + a_0$ of degree $n \geq 1$. Then
\begin{equation} \label{AimJacobiLemEq}
\left( \frac{f}{g} \right) = \psi(a_n)^{\max\{d(f),0\}} \psi \left( \mathrm{Res}(g,f) \right).
\end{equation} 
\end{lem}

\begin{proof}
 
Fix $f \neq 0$, and note that both sides above are completely multiplicative in $g$,
so we assume that $g$ is irreducible. 
For a root $\theta$ of $g$ we have
\begin{equation} \label{GaloisEq}
\mathrm{Res}(g,f) = a_n^{d(f)}\prod_{i=0}^{d(g)-1}f\big(\theta^{q^{i}}\big) = a_n^{d(f)}\prod_{i=0}^{d(g)-1}f(\theta)^{q^{i}} = a_n^{d(f)}f(\theta)^{\frac{q^{d(g)}-1}{q-1}}
\end{equation}
so we get the mod $p$ congruence
\begin{equation}
\begin{split}
\psi \left( \mathrm{Res}(g,f) \right) &= \psi(a_n)^{d(f)} \psi \left( f(\theta)^{\frac{q^{d(g)}-1}{q-1}} \right) \\ 
&\equiv \psi(a_n)^{d(f)} f(\theta)^{\frac{q^{d(g)}-1}{2}} \equiv 
\psi(a_n)^{d(f)} \left(\frac{f}{g} \right)
\end{split}
\end{equation}
which suffices for the lemma.
\end{proof}

Given a $D \in \F_q[T]$ we write $\mathrm{rad}(D)$ for the product of primes that appear in the factorization of $D$,
and $\mathrm{rad}_1(D)$ for the product of primes that appear with odd multiplicity in the factorization of $D$. 
The derivative of $D$ (with respect to $T$) is denoted by $D'$.

The next lemma interprets the M\"{o}bius function (on an arithmetic progression) as a Dirichlet character that `depends only on the derivative'.

\begin{lem} \label{PreparingMobiusToArithProgLem}

Let $m,k \geq 0$, $d \geq 1$ be integers with $k \neq d + m$.
For $M \in \mathcal{M}_m$, $g \in \mathcal{M}_d$, and $a \in \mathcal{M}_k$ coprime to $M$, define the polynomials
\begin{equation} \label{DefDEEq}
D = M^2\left(g + \frac{a}{M} \right)', 
\ E = \frac{\mathrm{rad}(D)}{\gcd \left(M, \mathrm{rad}(D) \right)},  
\ E_1 = \frac{\mathrm{rad}_1(D)}{\gcd \left(M, \mathrm{rad}_1(D) \right)}.
\end{equation}
Then
\begin{equation} \label{MobiusOnArithProgEq}
\mu(a + gM) = S \cdot \chi(w + g)
\end{equation}
where $w = w_{a,M,g'} \in \F_q[T]$, $\chi = \chi_{a,M,g'}$ is a (real) multiplicative character mod $E$ with conductor $E_1$,
and 
\begin{equation}
S = S_{d,a,M,g'} \in \{0,1,-1\}
\end{equation} 
with $S = 0$ if and only if $D = 0$.

\end{lem}

\begin{proof}

Pellet's formula (see \cite[Lemma 4.1]{Con05}) gives
\begin{equation} \label{PelletEq}
\mu(a + gM) = (-1)^{d(a + gM)}\psi \left( \mathrm{Disc}(a + gM) \right)
\end{equation}
and since $d(a + gM) = \max\{k, d+m\}$, we see that $(-1)^{d(a + gM)}$ can be absorbed into $S$.
Our assumption that $k \neq d + m$ implies $a + gM$ is monic,
so $\psi \left( \mathrm{Disc}(a + gM) \right)$ equals, up to a sign that $S$ absorbs,
\begin{equation}
\psi \left( \mathrm{Res}(a + gM, a' + g'M + gM') \right).
\end{equation}
By \cref{JacobiLem}, and the fact that $a + gM$ is monic, the above equals
\begin{equation}
\left( \frac{a' + g'M + gM'}{a + gM} \right)
\end{equation}
and since $\gcd(a,M) = 1$ by multiplicativity, this equals
\begin{equation}
\left( \frac{a' M + g'M^2 + gM M'}{a + gM} \right)\
\left( \frac{M}{a + gM} \right)^{-1} .
\end{equation}

Using quadratic reciprocity, we can absorb $\left( \frac{M}{a + gM} \right)$ into $S$.
Subtracting $M'(a + gM)$ from the numerator of the first Jacobi symbol above, we get (in the notation of equation \eqref{DefDEEq})
\begin{equation}
\left( \frac{D}{a + gM} \right)
\end{equation}
and set $w = 0$, $S = 0$ in case $D=0$.
Otherwise (if $D \neq 0$) we apply quadratic reciprocity once again to obtain
\begin{equation}
\left( \frac{a + gM}{D} \right)
\end{equation}
up to a sign that goes into $S$.

We write $\xi(a + gM)$ for the Jacobi symbol above, 
so that $\xi$ is a multiplicative character mod $\mathrm{rad}(D)$ with conductor $\mathrm{rad}_1(D)$.
Since $\mathrm{rad}(D)$ is squarefree, 
we see that $M$ is coprime to $E$ (from equation \eqref{DefDEEq}).
Therefore, by the Chinese remainder theorem, there exists a unique decomposition $\xi = \xi_E \xi_N$ to characters mod $E$ and $N = \gcd \left(M, \mathrm{rad}(D) \right)$ respectively.
In this notation, our Jacobi symbol equals $\xi_E(a + gM)\xi_N(a + gM)$
and the second factor is simply $\xi_N(a)$, so we immerse it in $S$.
Taking
$\overline M \in \F_q[T]$ with  $\overline M M \equiv 1 \ (E)$
we can write
\begin{equation}
\xi_E(a + gM) = \xi_E(a\overline M + g)\xi_E(M)
\end{equation}
and conclude by setting $w = a\overline M, \ \chi = \xi_E$, and dumping $\xi_E(M)$ into $S$.
\end{proof}

\begin{remark} \label{PreparingMobiusToArithProgRem}
With notation as in \cref{PreparingMobiusToArithProgLem}, suppose that $\chi$ is principal.
We can then write $D = AB^2$ for some $A,B \in \F_q[T]$ with $A \mid M$,
simply by taking $A = \mathrm{rad}_1(D)$ since principality gives $E_1 = 1$.
\end{remark}

\section{Linear forms in M\"{o}bius}

\begin{prop} \label{DerCongProp}

Let $p$ be a prime, let $q$ be a power of $p$, let $m,d$ be nonnegative integers, 
let $M \in \mathcal{M}_m$ be squarefree, and let  $a \in \mathbb{F}_q[T]$. Then
\begin{equation*}
\frac{\# \left\{h : h = g' \ \text{for some} \ g \in \mathcal{M}_d, \ h \equiv a \ \mathrm{mod} \ M \right\}}
{\# \left\{g' : g \in \mathcal{M}_d \right\}} \leq q^{- \min \left\{m, \left\lfloor \frac{d-1}{p} \right\rfloor \right\}}.
\end{equation*}

\end{prop}

\begin{proof}

Let $0 \leq j \leq p-1$ be the unique integer congruent to $d$ mod $p$.
Any $g \in \mathcal{M}_d$ can then be uniquely expressed as
\begin{equation}
g = \sum_{i=0}^{p-1} T^{i}g_i^p, \quad g_j \in \mathcal{M}_{\frac{d-j}{p}}, 
\quad d(g_i)  \leq \left\lfloor \frac{d - i}{p} \right\rfloor \ \text{for} \ i \neq j.
\end{equation}
For the derivative we then have
\begin{equation}
g' = \sum_{i=1}^{p-1} iT^{i-1}g_i^p
\end{equation}
so we set $a_i = iT^{i-1}$ for $1 \leq i \leq p-1$, and consider the congruence
\begin{equation}
\sum_{i=1}^{p-1}a_ig_i^p \equiv a \ \mathrm{mod} \ M.
\end{equation}

Since $M$ is squarefree, we can uniquely pick $b, b_1, \dots, b_{p-1} \in \F_q[T]$ with
\begin{equation}
b^p \equiv a \ \mathrm{mod} \ M, \quad b_i^p \equiv a_i \ \mathrm{mod} \ M, \ 1 \leq i \leq p-1,
\end{equation}
so the congruence above becomes
\begin{equation}
g_1 \equiv b - \sum_{i=2}^{p-1}b_ig_i \ \mathrm{mod} \ M
\end{equation}
and from this restriction on $g_1$ the proposition follows.
\end{proof}

The following technical proposition follows from the arguments of \cite[Page 371]{BGP92} or \cite[Section 9]{CG07} 
(see also \cite{BZ02}), 
which obtain stronger statements over $\mathbb{Z}$ in place of $\F_q[T]$.

\begin{prop} \label{BounNumSqArithProgProp}

Fix $\alpha, \epsilon > 0$, and a prime power $q$.
Then for integers $d,m,k \geq 0$ with $d \geq \epsilon (m+k)$,
and $M \in \mathcal{M}_m, \ A \in \mathcal{M}_k, \ a \in \F_q[T]$, we have
\begin{equation*} \label{DefiningAArithProgEq}
\# \left\{g \in \F_q[T] : d(g) < d, \ a+gM = \lambda AB^2, \quad \lambda \in \F_q, B \in \F_q[T] \right\} \ll q^{ (\frac{1}{2} + \alpha ) d} 
\end{equation*}
as $d \to \infty$, with the implied constant depending only on $\epsilon, \alpha$ and $q$.

\end{prop}

We can now prove \cref{SecRes}. 
We first prove the ``generic" special case where the derivatives of certain parameters are distinct, and then prove the general case. 

\begin{prop} \label{LinearFormsDistinctDerivatives} 

Fix $\epsilon, \delta> 0$,  $0 < \beta< 1/2$, and a positive integer $n$. Let $q$ be a power of an odd prime $p$ such that 
\begin{equation}\label{prop-q-condition} 
q > \left(  \frac{pne } {\min \left\{ \frac{\epsilon}{\epsilon + 2}, \frac{\epsilon \delta}{\epsilon + \delta} \right\}}\right)^{\frac{2}{ 1- 2\beta}}.
\end{equation}
Then for  nonnegative integers $d,m_1, \dots, m_n,k_1, \dots, k_n$ with 
\begin{equation}
d \geq \max\{\epsilon m_1, \dots, \epsilon m_n, \delta k_1, \dots, \delta k_n \}, \quad k_i \neq d+m_i, \ 1 \leq i \leq n,
\end{equation}
and pairs $(a_i, M_i) \in \mathcal{M}_{k_i} \times \mathcal{M}_{m_i}$ for $1 \leq i \leq n$ such that the derivatives $\left(\frac{a_i}{M_i}\right)'$ are all distinct, we have
\begin{equation}
\sum_{g \in \mathcal{M}_{d}} \prod_{i=1}^n \mu \left(a_i + gM_i \right) \ll |\mathcal{M}_d|^{1-  \frac{\beta}{p}}
\end{equation}
as $d \to \infty$, with the implied constant depending only on $\beta, \epsilon, \delta, n$ and $q$.\end{prop}

\begin{remark} Note that the statement of this proposition remains meaningful even if $\epsilon$ and $\delta$ are very large, though it is at its strongest when $\epsilon$ and $\delta$ are small. \end{remark}

\begin{proof} 

Let us first assume that for every $1 \leq i \leq n$ we have $\gcd(a_i, M_i) = 1$.

We say that $g_1,g_2 \in \mathcal{M}_{d}$ are equivalent if $g_1' = g_2'$,
and let $\mathcal{R}$ be a complete set of representatives of equivalence classes.
So for each $g \in \mathcal{M}_{d}$ there exists a unique $r \in \mathcal{R}$ such that $(g-r)' = 0$,
and therefore also a unique $s \in \F_q[T]$ such that $g - r = s^p$.
We can thus write our sum as
\begin{equation} \label{rsSum}
\sum_{r \in \mathcal{R}} \sum_{d(s) < t} 
\prod_{i=1}^n \mu \left(a_i +  (r + s^p)M_i \right), 
\quad t = \frac{d}{p}.
\end{equation}
By \cref{PreparingMobiusToArithProgLem} (the notation of which is used throughout), our sum equals
\begin{equation} \label{SrSumEq}
\sum_{r \in \mathcal{R}} \prod_{i=1}^n S_{r}^{(i)} 
\sum_{\substack{d(s) < t}} \prod_{i = 1}^n \chi_r^{(i)} \left( w_r^{(i)} + s^p \right)
\end{equation}
with $\chi_r^{(i)}$ a character to a squarefree modulus $E_r^{(i)}$ 
(defined in \cref{PreparingMobiusToArithProgLem} using $D_r^{(i)}$).
Hence, there exist $f_r^{(i)} \in \F_q[T]$ with 
\begin{equation} \label{prootexteq}
{f_r^{(i)}}^p \equiv w_r^{(i)} \ \mathrm{mod} \ E_r^{(i)}, \quad 1 \leq i \leq n,
\end{equation}
so using the fact that $\chi_r^{(i)}$ is real, we see that our inner sum equals
\begin{equation} 
\sum_{d(s) < t} \prod_{i=1}^n \chi_r^{(i)} \left(f_r^{(i)} + s \right).
\end{equation}

For $1 \leq i < j \leq n$ we set
\begin{equation}
G_r^{(i,j)} = \gcd \left( E_r^{(i)}, E_r^{(j)} \right), \quad U_r = \lcm_{i,j} \left(  G_r^{(i,j)}  \right ), 
\quad \ell_r = d(U_r),
\end{equation}
and claim that, for every integer $\ell \geq 0$ and $\gamma > 0$, we have
\begin{equation}\label{ell-counting-claim}
\# \left\{r \in \mathcal{R} : \ell_r \geq \ell \right\}
\ll q^{d \left( 1 - \frac{1}{p}  + \gamma \right) - \min \left( \frac{d-1}{p} , \ell \right)}, \quad d \to \infty.
\end{equation}

To do this, first note that, for any $1 \leq i < j \leq n$ , since $G_r^{(i,j)}$ divides $E_r^{(i)}$, and divides $E_r^{(j)}$,
it also divides $D_r^{(i)}$, and divides $D_r^{(j)}$ so therefore it divides the polynomial 
\begin{equation}
M_j^2 D_r^{(i)} - M_i^2 D_r^{(j)} = 
M_i^2 M_j^2 \left( \frac{a_i}{M_i} \right)' - M_i^2 M_j^2 \left( \frac{a_j}{M_j}\right)'
\end{equation}
which is nonzero by our initial assumption.  
The degree of the above polynomial is at most $ d/\delta + 3d/\epsilon$,
so by the divisor bound (see \cite[Equation 1.81]{IK04}),
the polynomial $G_{r}^{(i,j)}$ attains $ \ll q^{2 \gamma d/ (n(n-1))}$ values, for any $\gamma > 0$. 
Hence, the tuple $\left(  G_{r}^{(i,j)}\right)_{1\leq i < j \leq n} $ attains $\ll q^{ \gamma d}$ values. For each possible tuple $G_{*}^{(i,j)}$, we can recover the residue class of $r'$ mod $ G_{*}^{(i,j)}$ from the congruence
\begin{equation}
M_j^2 r'   \equiv  M_j'a_{j} - a_{j}'M_j \ \mathrm{mod} \ G_r^{(i,j)}.
\end{equation} 
and the fact that  $G_r^{(i,j)}$ is prime to $M_j$ (because $D_r^{(j)}$ is prime to $M_j$, by definition).
Combining these for all $i,j$ we can also recover the residue class of $r'$ mod the least common multiple $U_*$ of $G_{*}^{(i,j)}$.

\cref{DerCongProp} tells us that for any $\alpha \in \F_q[T]$ we have
\begin{equation*}
\# \{r \in \mathcal{R} : r' \equiv \alpha \ \mathrm{mod} \ U_* \} \leq
q^{- \min \left\{d(U_*), \left\lfloor \frac{d-1}{p} \right\rfloor \right\}} \# \mathcal{R}
\ll  q^{d \left( 1 - \frac{1}{p}  \right) - \min \left( \frac{d-1}{p} , \ell \right)  }
\end{equation*}
since $d(U_*) \geq \ell$, so our claim is established.

Next,
we observe from \cref{ell-counting-claim} that the contribution of those $r \in \mathcal{R}$ with $\ell_r \geq \frac{d-1}{p}$
to \cref{rsSum} is $\ll q^{ d \left( 1 - \frac{1}{p} + \gamma \right)}$ and thus can be ignored as we can choose $\gamma$ small enough that
$1 - \frac{1}{p} + \gamma \leq  1- \frac{\beta}{p}$. 
So we may assume that $\ell_r < \frac{d-1}{p}$.

We further set
\begin{equation}
\widetilde{E}_r^{(i)}  = \gcd \left(E_r^{(i)}, U_r \right), \quad \widehat{E}_r^{(i)}  = \frac{E_r^{(i)}}{\widetilde{E}_r^{(i)}}, 
\quad 1 \leq i \leq n,
\end{equation}
and note that $\gcd(\widetilde{E}_r^{(i)}, \widehat{E}_r^{(i)}) = 1$ as $E_r^{(i)}$ is squarefree.
The Chinese remainder theorem then gives a unique decomposition
\begin{equation}
\chi_r^{(i)} = \widetilde{\chi}_r^{(i)} \widehat{\chi}_r^{(i)}, \quad 1 \leq i \leq n,
\end{equation}
to characters mod $\widetilde{E}_r^{(i)}$ and $\widehat{E}_r^{(i)}$ respectively.
In this notation, our sum reads
\begin{equation} 
\sum_{d(s) < t} \prod_{i=1}^n \widetilde{\chi}^{(i)}_r \left(f_r^{(i)} + s \right) 
\prod_{i=1}^n \widehat{\chi}^{(i)}_r \left(f_r^{(i)} + s \right)
\end{equation}
so splitting according to the residue class $u$ of $s$ mod $U_r$ we get
\begin{equation} 
\sum_{d(u) < \ell_r}\prod_{i=1}^n \widetilde{\chi}^{(i)}_r \left(f_r^{(i)} + u \right)  
\sum_{d(h) < t - \ell_r} \prod_{i=1}^n \widehat{\chi}^{(i)}_r \left(f_r^{(i)} + u + hU_r\right).
\end{equation}

Since $\gcd(\widetilde{E}_r^{(i)}, \widehat{E}_r^{(i)}) = 1$, we get that $\gcd(U_r, \widehat{E}_r^{(i)}) = 1$ for $1 \leq i \leq n$.
Hence, there exist $V^{(i)}_r \in \F_q[T]$ with $U_rV_r^{(i)} \equiv 1 \ \mathrm{mod} \ \widehat{E}_r^{(i)}$.
Summing trivially over $u$, we may thus consider
\begin{equation} 
\prod_{i=1}^n \widehat{\chi}_r^{(i)}(U_r)
\sum_{d(h) < t - \ell_r} \prod_{i=1}^n \widehat{\chi}^{(i)}_r \left(f_r^{(i)}V_r^{(i)} + uV_r^{(i)} + h\right). 
\end{equation}
From $\gcd(\widetilde{E}_r^{(i)}, \widehat{E}_r^{(i)}) = 1$ 
we moreover conclude that $\{\widehat{E}_r^{(i)}\}_{i = 1}^n$ are pairwise coprime.
The Chinese remainder theorem then gives an $f_{r,u} \in \F_q[T]$ with
\begin{equation}
f_{r,u} \equiv f_r^{(i)}V_r^{(i)} + uV_r^{(i)} \ \mathrm{mod} \ \widehat{E}_r^{(i)}, \quad 1 \leq i \leq n,
\end{equation}
so defining the character $\chi_r = \widehat{\chi}_r^{(1)} \cdots \widehat{\chi}_r^{(n)}$, 
mod $E_ r = \widehat{E}_r^{(1)} \cdots \widehat{E}_r^{(n)},$
the sum above becomes
\begin{equation} \label{AdKanEq}
\sum_{d(h) < t - \ell_r} \chi_r (f_{r,u} + h). 
\end{equation}
 
We have
\begin{equation} \label{InsteadEq}
\frac{t}{d(E_r)}  \geq \frac{t   }{n \max \{ d+ 2d/\epsilon, d/\delta +  d/\epsilon\}} = 
\frac{\min \left\{ \frac{\epsilon}{\epsilon + 2}, \frac{\epsilon \delta}{\epsilon + \delta} \right\}}{ pn } \eqdef \eta
\end{equation} 
so if $\chi_r$ is nonprincipal, \cref{FFCSCor} bounds the sum above by $\ll q^{  (1-\beta') t } $, 
for some $\beta' > \beta + p \gamma$, with $\gamma>0$ arbitrarily small.  
Because we have 
\begin{equation}
q > \left(  \frac{p n e  } {\min \left\{ \frac{\epsilon}{\epsilon + 2}, \frac{\epsilon \delta}{\epsilon + \delta} \right\}}\right)^{\frac{2}{ 1- 2\beta}}
\end{equation}
by assumption, 
using our definition of $\eta$ this can be writeen as
\begin{equation}
q >\left(e \eta^{-1}\right)^{ \frac{ 2}{1-2 \beta} }
\end{equation}
and therefore 
\begin{equation}
q > \left( e \eta^{-1} \right)^{ \frac{ 2}{1-2 \beta'} }
\end{equation}
for any sufficiently small choice of $\gamma$.

Hence, those $r \in \mathcal{R}$ for which $\ell_r =\ell $ and $\chi_r$ is nonprincipal, 
contribute $ \ll q^{ (1-\beta' ) t  } q^{\ell}$ individually, so the total contribution to our initial sum is $\ll$
\begin{equation*}
q^{ (1-\beta') t  +\ell}  q^{d \left( 1 - \frac{1}{p}  + \gamma \right) - \ell }   =  q^{ (1 - \beta) t} q^{ d \left( 1- \frac{1}{p} \right) } q^{ \gamma d - (\beta'- \beta) t }  = q^{ d \left( 1 -\frac{\beta}{p} \right) } q ^{ d ( \gamma - (\beta'-\beta) / p) } .
\end{equation*}
The increase from summing over the possible values of $\ell$ is linear in $d$ and thus can be bounded by the exponential $q ^{ d (  (\beta'-\beta) / p - \gamma) }  $, so the contribution of all the terms where $\chi$ is nonprincipal is $\ll q^{ d  \left(1 -\frac{\beta}{p} \right)}$.

Let $r \in \mathcal{R}$ for which $\chi_r$ is principal. 
Pairwise coprimality implies that $\widehat{\chi}_r^{(1)}$ is principal as well,
so the conductor of $\chi_r^{(1)}$ divides that of $\widetilde{\chi}_r^{(1)}$,
and the latter divides $U_r$. 
Thus, for some $\lambda \in \F_q$ and monic $B \in \F_q[T]$ we have
\begin{equation*}
M_1^2r' +  a_1'M_1 - a_1M_1' = D_r^{(1)} =  \mathrm{rad}_1 \left( D_r^{(1)} \right) \lambda B^2 
= \lambda A \overline{A} B^2, \ A \mid M_1, \overline{A} \mid U_r
\end{equation*}
where $A$ is the greatest common divisor of $M_1$ and $\mathrm{rad}_1 (D_r^{(1)})$,
and $\overline{A}$ is the conductor of $\chi_r^{(1)}$.
Since $U_r$ divides the nonzero polynomial $Q$ defined to be
\begin{equation}
\prod_{1 \leq i < j \leq n} M_i^2 M_j^2 \left[ \left( \frac{a_i}{M_i}\right)' - \left( \frac{a_j}{M_j}\right)' \ \right],
\end{equation}
it follows that $\overline{A} \mid Q$ as well, 
so \cref{BounNumSqArithProgProp} ensures that the number of $r$ for which the equation above holds is 
\begin{equation}
\ll  d_2(M_1Q)q^{( 1/2 + \zeta) d}.
\end{equation}
The divisor bound then allows us to neglect those $r$ for which $\chi_r$ is principal, as long as $\frac{1}{2} + \zeta + \frac{1}{p} < 1 - \frac{\beta}{p}$, which is alright as $\beta < \frac{1}{2} \leq p \left( \frac{1}{2}- \frac{1}{p} \right) $ so we can choose $\zeta$ small enough that this inequality holds.

Let us now handle the case when $a_i$ and $M_i$ are not coprime. 
Set 
\begin{equation}
H_i= \gcd(a_i, M_i), \quad 1 \leq i \leq n.
\end{equation} 
If some $H_i$ is not squarefree, the sum vanishes and the bound is trivial. 
Otherwise, we have the identity 
\begin{equation} 
\mu( a_i + g M_i ) = 
\begin{cases} 
\mu\left( \frac{ a_i}{H_i } + g \frac{ M_i}{ H_i} \right) \mu( H_i)  &\textrm{if } \gcd \left(  \frac{ a_i}{H_i  } + g \frac{ M_i}{ H_i },H_i \right) =1  \\ 0 & \textrm{if }\gcd \left(  \frac{ a_i}{H_i  } + g \frac{ M_i}{ H_i },H_i \right) \neq 1.
\end{cases} 
\end{equation}

Thus we can write the sum as the constant factor $\prod_{i=1}^{n} \mu(H_i)$ times a similar sum, except that the degrees of $a_i$ and $M_i$ are reduced and the terms where $ \gcd \left(  \frac{ a_i}{H_i  } + g \frac{ M_i}{ H_i },H_i \right) \neq 1$ are removed. 
The degree reduction preserves our inequalities and so is no trouble. 
Removing the terms with $ \gcd \left(  \frac{ a_i}{H_i  } + g \frac{ M_i}{ H_i },H_i \right) \neq 1$ amounts to removing those $g$ which lie in a particular residue class modulo each prime factor of $H_i$, for a total of at most $\Gamma = \sum_{i=1}^n \omega(H_i)$ residue classes.

We perform the same argument to this restricted sum. 
The only change that occurs is when we write $g = r+ s^p$, we must assume that $s$ avoids a corresponding set of residue classes modulo these primes. 
By inclusion-exclusion, 
we can write a Dirichlet character sum avoiding $\Gamma$ residue classes as an alternating sum of Dirichlet character sums in at most $2^{\Gamma}$ residue classes, 
and hence as an alternating sum of at most $2^{\Gamma}$ shorter Dirichlet character sums. 
Because the sums over each residue class are shorter, we can get the same bound for them by \cref{FFCSCor}. 
Thus our final bound for this case is worse by a factor of 
\begin{equation}
2^{\Gamma} = q^{ \sum_{i=1}^n o ( d( H_i)) } = q^{ o(d)}. 
\end{equation}
We can absorb this into our bound by slightly increasing $\beta$ so that it still satisfies the strict inequality \eqref{prop-q-condition}.
\end{proof}

\begin{thm} \label{LinearFormsMobThm}

Fix $\epsilon, \delta> 0$,  $0 < \beta< 1/2$, and a positive integer $n$. Let $q$ be a power of an odd prime $p$ such that 
\begin{equation} \label{StrictIn}
q > \left(  \frac{pne} {\min \left\{ \frac{\epsilon}{\epsilon + 2}, \frac{\epsilon \delta}{\epsilon + \delta} \right\}}\right)^{\frac{2}{ 1- 2\beta}}.
\end{equation}
Then for  nonnegative integers $d,m_1, \dots, m_n,k_1, \dots, k_n$ with 
\begin{equation}
d \geq \max\{\epsilon m_1, \dots, \epsilon m_n, \delta k_1, \dots, \delta k_n \}, \quad k_i \neq d+m_i, \ 1 \leq i \leq n,
\end{equation}
and pairs $(a_i, M_i) \in \mathcal{M}_{k_i} \times \mathcal{M}_{m_i}$ for $1 \leq i \leq n$ with $a_i/M_i$ distinct, we have
\begin{equation}
\sum_{g \in \mathcal{M}_{d}} \prod_{i=1}^n \mu \left(a_i + gM_i \right) \ll |\mathcal{M}_d|^{1 - \frac{\beta}{p} } 
\end{equation}
as $d \to \infty$, with the implied constant depending only on $\epsilon, \delta, \beta, n$ and $q$.

\end{thm}

\begin{proof}

Our initial assumption is that there are no coincidences among $(a_i/M_i)$, for $1 \leq i \leq n$,
so we can find a prime $P$ not dividing 
\begin{equation}
\prod_{i=1}^n M_i \prod_{1 \leq i < j \leq n} a_iM_j - a_jM_i,
\end{equation}
with $d(P) = o(d)$. Let $t'=d(P)$.
Splitting our initial sum according to the residue class $z$ of $g$ mod $P$ we get
\begin{equation}
\sum_{d(z) < t'} \sum_{f \in \mathcal{M}_{d-t'}} \prod_{i=1}^n \mu \left(a_i + zM_i + fM_iP \right).
\end{equation}

We can bound the sums over residue classes by applying \cref{LinearFormsDistinctDerivatives}. Indeed, if for some $1 \leq i < j \leq n$ we have
\begin{equation}
\left( \frac{a_i + zM_i}{M_iP}\right)' = \left( \frac{a_j + zM_j}{M_jP}\right)'
\end{equation}
then we get that $ \ a_iM_j \equiv a_jM_i  \ \mathrm{mod} \ P$, a contradiction. Because the length of the sum in this case is $q^{ d -t'  } $, we obtain a savings in each term of $q^{ (d- t') \beta/2p }$ from \cref{LinearFormsDistinctDerivatives}. To obtain our desired savings of $ q^{ d \beta/2p}$, we must choose $\beta+o(1)$ instead of $\beta$ in the statement of \cref{LinearFormsDistinctDerivatives} . Similarly to ensure that $d- t' \geq \max\{\epsilon m_1, \dots, \epsilon m_n, \delta k_1, \dots, \delta k_n \}$ we must choose $\epsilon-o(1)$ and $\delta-o(1) $. 
However because the inequality from \cref{StrictIn} is strict, 
we may increase $\beta$ by $o(1)$ and reduce $\epsilon$ and $\delta$ by $o(1)$ in such a way that this inequality is still satisfied.
\end{proof}


We prove two corollaries that give weaker results under conditions that are simpler to state.

\begin{cor} \label{MobThmLittleO}

Fix $\epsilon, \delta> 0$ and a positive integer $n$. Let $q$ be a power of an odd prime $p$ such that 
\begin{equation}
q >   p^2n^2e^2 \max\left(  1+\frac{2}{\epsilon} , \frac{1}{\epsilon} + \frac{1}{ \delta} \right)^2.
\end{equation}
Then for  nonnegative integers $d,m_1, \dots, m_n,k_1, \dots, k_n$ with 
\begin{equation}
d \geq \max\{\epsilon m_1, \dots, \epsilon m_n, \delta k_1, \dots, \delta k_n \}, \quad k_i \neq d+m_i, \ 1 \leq i \leq n,
\end{equation}
and distinct coprime pairs $(a_i, M_i) \in \mathcal{M}_{k_i} \times \mathcal{M}_{m_i}$ for $1 \leq i \leq n$, we have
\begin{equation}
\sum_{g \in \mathcal{M}_{d}} \prod_{i=1}^n \mu \left(a_i + gM_i \right) =o \left(  |\mathcal{M}_d|\right)  \end{equation}
as $d \to \infty$ for fixed $\epsilon, \delta, n, q$.

\end{cor}

\begin{proof}
The assumed lower bound on $q$ is equivalent to
\begin{equation}
q > \left(  \frac{pne} {\min \left\{ \frac{\epsilon}{\epsilon + 2}, \frac{\epsilon \delta}{\epsilon + \delta} \right\}}\right)^{2}.
\end{equation}
If $q$ satisfies this inequality, we can take $\beta$ small enough that the inequality
\begin{equation}
q > \left(  \frac{pne} {\min \left\{ \frac{\epsilon}{\epsilon + 2}, \frac{\epsilon \delta}{\epsilon + \delta} \right\}}\right)^{\frac{2}{ 1- 2\beta}}
\end{equation}
holds, and then apply \cref{LinearFormsMobThm}. 
\end{proof}

\begin{cor} \label{ChowlaThm} 
Let $q$ be a power of an odd prime $p$ such that 
\begin{equation}
q >   p^2n^2e^2,
\end{equation}
and let $(a_i, M_i) \in \mathcal{M}_{k_i} \times \mathcal{M}_{m_i}$ be distinct coprime pairs for $1 \leq i \leq n$.
Then
\begin{equation}
\sum_{g \in \mathcal{M}_{d}} \prod_{i=1}^n \mu \left(a_i + gM_i \right) =o \left(  |\mathcal{M}_d|\right)  \end{equation}
as $d \to \infty$ for fixed $q,n, a_i, M_i$. \end{cor}

\begin{proof} 
We can take $\epsilon, \delta$ large enough that 
\begin{equation}
q >   p^2n^2e^2 \max\left(  1+\frac{2}{\epsilon} , \frac{1}{\epsilon} + \frac{1}{ \delta} \right)^2.
\end{equation}
For this $\epsilon$ and $\delta$, the inequalities 
\begin{equation}
d \geq \max\{\epsilon m_1, \dots, \epsilon m_n, \delta k_1, \dots, \delta k_n \}, \quad k_i \neq d+m_i, \ 1 \leq i \leq n,
\end{equation}
will be satisfied for all $d$ sufficiently large. We can then apply \cref{MobThmLittleO} to deduce the claim. \end{proof} 

\section{Level of distribution}

The following will be obtained using the techniques of \cite{FM98} and \cite{FKM14} in the appendix.

\begin{thm} \label{FKMThm}
Fix an odd prime power $q$. 
Then for any $\theta>0$, for nonnegative integers $d,m$ with $d \leq m$,
squarefree $M \in \mathcal{M}_m$, and an additive character $\psi$ mod $M$, we have
\begin{equation} \label{FMeq}
\sum_{\substack{g \in \mathcal{M}_d \\ (g, M) = 1}} \mu(g)\psi(\overline g) \ll q^{ \left(\frac{3}{16}+ \theta\right)m + \frac{25}{32}d }
\end{equation}
as $d \to \infty$, with the implied constant depending only on $q$ and $\theta$.

\end{thm}

The following proposition allows us to identify the main term in sums of von Mangoldt in arithmetic progressions.

\begin{prop} \label{MainTermPrimesLD}

Fix a prime power $q$. 
For nonnegative integers $d,m$, and a squarefree $M \in \mathcal{M}_m$ we have
\begin{equation}
\sum_{k = 1}^{d} kq^{-k} \sum_{\substack{A \in \mathcal{M}_k \\ (A,M) = 1}} \mu(A) = 
- \frac{q^m}{\varphi(M)} + q^{o(m + d) - d}.
\end{equation}

\end{prop}

\begin{proof}

The left hand side above is the sum of the first $d$ coefficients of the power series
\begin{equation*}
\begin{split}
u\frac{ d}{du} \sum_{k = 1}^{\infty}  q^{-k} u^k \sum_{\substack{A \in \mathcal{M}_k \\ (A,M) = 1}} \mu(A) &= 
u\frac{d}{du} \prod_{P \nmid M} \left( 1 - u^{d(P)} |P|^{-1} \right) \\
&= u \frac{d}{du} \left(  (1-u) \prod_{P \mid M}  \left(1- u^{d(P)} |P|^{-1} \right)^{-1}\right).
\end{split}
\end{equation*}

Summing all the coefficients of a power series, evaluates it at $u=1$.
Hence, the main term comes from the equality
\begin{equation*}
\left( u\frac{ d}{du} \left((1-u) F(u)\right) \right) (1) = -F(1), \quad F(u) = \prod_{P \mid M}  \left(1- u^{d(P)} |P|^{-1} \right)^{-1}.
\end{equation*}

The coefficients of degree greater than $d$ contribute to the error term.  
To bound the sum of these coefficients, we can write the degree $k$ coefficient as 
\begin{equation}
k \oint\limits_{|u| = r} \frac{  (1-u) F(u) }{u^{k+1} }
\end{equation}
for $r < q$, getting a bound of 
\begin{equation}
(1+r) r^{-k-1} \max_{|u| = r} F(u).
\end{equation}
As long as the above maximum is subexponential in $m$ for all $r<q$, the expression above will be $q^{o(m)}  (1+r ) r^{-k-1}$, so the sum of the coefficients of degree greater than $d$ is 
\begin{equation}
q^{o(m)}  \sum_{k>d} k (1+r ) r^{-k-1} = q^{ o (m)}  (r-o(1))^{-d}. 
\end{equation}
Taking $r$ arbitrarily close to $q$, the above is $q^{ o(m) } (q-o(1))^d = q^{ o(m+d) - d}$. 

The value of $F(u)$ is indeed subexponential in $m$, because $M$ has $o(m)$ prime divisors and each contributes at most 
\begin{equation}
\left(1 -\left( \frac{r}{ q}\right)^{d(P)} \right) ^{-1} \leq \left( 1- \frac{r}{q} \right)^{-1}.
\end{equation}

\end{proof}

In the following we deduce, from our results on the M\"{o}bius function, 
a level of distribution beyond $1/2$ for primes in arithmetic progressions to squarefree moduli.
We shall use the convolution identity $\Lambda = -1 * (\mu \cdot \deg)$
which for $f \in \F_q[T]^+$ of degree $d \geq 0$ says that
\begin{equation} \label{ConvId}
\Lambda(f) = - \sum_{k=1}^{d} k \mathop{\sum_{A \in \mathcal{M}_k} \sum_{B \in \mathcal{M}_{d-k}}}_{AB = f} \mu(A).
\end{equation}

\begin{cor} \label{PrimesLODCor}

For any $0< \omega < 1/32$, for any odd prime $p$ and power $q$ of $p$ such that $q> p^2e^2  \left( 1+ \frac{50}{   1- 32 \omega } \right)^2 $,
the following holds.
For nonnegative integers $d,m$ with $d \geq (1 - \omega)m , $, 
squarefree $M \in \mathcal{M}_m$, 
and $a \in \mathbb F_q[T]$ with $d(a)<d+m$ and coprime to $M$, we have
\begin{equation}
\sum_{g \in \mathcal{M}_d} \Lambda(a + gM) = \frac{q^{d + m}}{\varphi(M)} + E_q
\end{equation}
as $d \to \infty$, with a power saving error term $E_q$.

\end{cor}

\begin{proof}

We can assume $\ell = m$, and use \cref{ConvId} to write our sum as
\begin{equation}
\sum_{\substack{f \in \mathcal{M}_{m+d} \\ f \equiv a \ \mathrm{mod} \ M}} \Lambda(f) = 
- \sum_{k = 1}^{d+m} k \sum_{\substack{A \in \mathcal{M}_k \\ (A,M) = 1}} \mu(A) \sum_{\substack{B \in \mathcal{M}_{m+d-k} \\ AB \equiv a \ \mathrm{mod} \ M}} 1
\end{equation}
so by \cref{MainTermPrimesLD} the range $k \leq d$ contributes the main term.
%

The (absolute value of the) contribution of any $k>d$ is
\begin{equation}
\begin{split} 
&q^{d-k} \left| \sum_{\substack{A \in \mathcal{M}_k \\ (A,M) = 1}} \mu(A)  \sum_{\substack{ \psi \colon \mathbb F_q[T]/ M \to \mathbb C^\times \\ \psi(f) =0 \textrm{ if } d(f)<m+d-k} } \psi \left(a \overline A \right)  \overline{ \psi  ( T^{m+d-k})} \right| \leq \\
&q^{d-k} \sum_{\substack{ \psi \colon \mathbb F_q[T]/ M \to \mathbb C^\times \\ \psi(f) =0 \textrm{ if } d(f)<m+d-k} }  \left|   \sum_{\substack{A \in \mathcal{M}_k \\ (A,M) = 1}} \mu(A)   \psi \left(a \overline A \right) \right|
\end{split}
\end{equation} 
so by \cref{FKMThm} we get
\begin{equation} 
\ll q^{d-k} q^{k-d} q^{ \left(\frac{3}{16}+ \theta\right)m + \frac{25}{32}k } = q^{ \left(\frac{3}{16}+ \theta\right)m + \frac{25}{32}k }
\end{equation}
which gives a power saving as long as \begin{equation}\label{primes-range-splitting}  \left(\frac{3}{16}+ 2\theta\right)m + \frac{25}{32}k < d.\end{equation}

The (absolute value of the) contribution of any other $k$ is at most
\begin{equation}  
\sum_{\substack{B \in \mathcal{M}_{m+d-k} \\ (B,M)=1}} \left| \sum_{\substack{A \in \mathcal{M}_k \\ AB \equiv a \ \mathrm{mod} \ M} } \mu(A)\right| =   \sum_{\substack{B \in \mathcal{M}_{m+d-k} \\ (B,M)=1}} \left| \sum_{g \in \mathcal{M}_{k-m} } \mu(b+g M )\right| 
\end{equation}
where $b$ is the unique monic polynomial of degree $m$ congruent to $a B^{-1}$ modulo $m$.  We apply \cref{LinearFormsMobThm} with some fixed $\beta>0$ and with \begin{equation} \epsilon = \delta =  \frac{32}{25} \left(   \frac{1}{32} - \omega  - 2 \theta  \right).  \end{equation} 
Then we because \eqref{primes-range-splitting} does not hold, we have
\begin{equation*} 
k - m \geq \frac{32}{25} \left( d -  \left(\frac{31}{32}+ 2\theta\right)m \right) \geq 
\frac{32}{25} \left( (1-\omega) m -  \left(\frac{31}{32}+ 2\theta\right)m \right) = \epsilon m 
\end{equation*} 
so the conditions of \cref{LinearFormsMobThm} are satisfied as long as 
\begin{equation} 
q> \left( pe  \left( 1+ \frac{25}{16} \frac{1}{   \frac{1}{32} - \omega  - 2 \theta } \right) \right)^{ \frac{2}{ 1-2 \beta}}. 
\end{equation}
Because we may take $\beta$ and $\theta$ arbitrarily small, it suffices to have 
\begin{equation} q>  p^2e^2  \left( 1+ \frac{25}{16} \frac{1}{   \frac{1}{32} - \omega } \right)^2 = p^2e^2  \left( 1+ \frac{50}{   1- 32 \omega } \right)^2 \end{equation}

Summation over $k$ gives only an extra logarithmic factor, so it preserves our power savings.
\end{proof}

\section{Twins}

\subsection{Chowla sums over primes}

We establish cancellation in M\"{o}bius autocorrelation over primes.

\begin{cor} \label{MobPrimCor}

Fix $\tilde{\epsilon}, \tilde{\delta}> 0$,  $0< \alpha<1$, $0 < \beta< 1/2$, and a positive integer $n$. Let $q$ be a power of an odd prime $p$ such that \[q > \left(  p(n+1) e \max \left( 1+ \frac{2+ 2\alpha + 4 \tilde{\epsilon}^{-1} }{1-\alpha} , \frac{1 + \alpha + 2 \tilde{\epsilon}^{-1} + 2 \tilde{\delta}^{-1} }{ 1-\alpha}   \right) \right)^{\frac{2}{ 1- 2\beta}}. \]
Take nonnegative integers $d,m_1, \dots, m_n,k_1, \dots, k_n$ with 
\begin{equation}
d \geq \max\{\tilde{\epsilon} m_1, \dots, \tilde{\epsilon} m_n, \tilde{\delta} k_1, \dots, \tilde{\delta} k_n \}, \quad k_i \neq d+m_i, \ 1 \leq i \leq n,
\end{equation}
and distinct coprime pairs $(a_i, M_i) \in \mathcal{M}_{k_i} \times \mathcal{M}_{m_i}$ for $1 \leq i \leq n$.

Furthermore let $m,k$ be nonnegative integers with $m \leq \alpha d, \ k < m+d$,
let $M \in \mathcal{M}_m$, 
and let $a \in \mathcal{M}_k$ with $(a,M)$ distinct from $(a_i,M_i)$ for every $1 \leq i \leq n$.
Then
\begin{equation}
\sum_{g \in \mathcal{M}_{d}} \Lambda(a + gM) \prod_{i=1}^n \mu \left(a_i + gM_i \right) \ll | \mathcal{M}_d|^{1 - \frac{\beta (1-\alpha) }{2p} } 
\end{equation}
as $d \to \infty$, with the implied constant depending only on $\epsilon, \delta, \alpha, n$ and $q$. 

\end{cor}

\begin{proof}

We can assume $(a,M) = 1$, set $z = \left\lfloor \frac{d+m}{2} \right\rfloor$, and use \cref{ConvId} to write
\begin{equation*}
\Lambda(a + gM) = - \sum_{b=1}^{z} b \sum_{\substack{ B \in \mathcal{M}_b \\ B \mid a + gM }} \mu(B) - 
\sum_{c=0}^{d + m - z - 1} (d + m - c) \sum_{\substack{C \in \mathcal{M}_c \\ C \mid a+gM}} \mu\left(\frac{a+gM}{C} \right).
\end{equation*}
For every $B$ above, taking a monic $N = N(B) \in \F_q[T]$ with 
\begin{equation}
MN \equiv -a \ \mathrm{mod} \ B, \quad \ d(N) \neq k_i - m_i, \quad d(N) \leq b + n,
\end{equation}
and writing $g = N + hB$ with $h \in \mathcal{M}_{d-b}$,
we see that the sum over $b$ contributes
\begin{equation}
\sum_{b=1}^{\left\lfloor \frac{d+m}{2} \right\rfloor} b \sum_{B \in \mathcal{M}_b} \mu(B) \sum_{h \in \mathcal{M}_{d-b}} 
\prod_{i=1}^n \mu \left(a_i + NM_i + hBM_i \right)
\end{equation}
so we can apply \cref{LinearFormsMobThm} to the innermost sum, taking 
\begin{equation} 
\epsilon = \frac{1 - \alpha}{ 1+ \alpha + 2 \tilde{\epsilon}^{-1}} \leq \frac{d-m }{d+m+2m_i} = \frac{d - \frac{d+m}{2} } { \frac{d+m}{2} + m_i }  \leq \frac{ d-b} { b + m_i}  
\end{equation}
and  
\begin{equation} 
\begin{split}
\delta = \min \left( \frac{1-\alpha} {2 \tilde{\delta}^{-1}}, \frac{1 - \alpha}{ 1+ \alpha + 2 \tilde{\epsilon}^{-1}} \right) &\leq 
\min \left( \frac{ d-b}{k_i}, \frac{d-b}{b+m_i} \right) \\
&\leq \min \left( \frac{d-b}{ \deg a_i}, \frac{d-b}{\deg N M_i} + o(1) \right).
\end{split}
\end{equation} 

To obtain the condition on $q$, note that, when calculating $\max( 1+ \frac{2}{\epsilon}, \frac{1}{\epsilon}+ \frac{1}{\delta})$, we can treat $\delta$ as $ \frac{1-\alpha} {2 \tilde{\delta}^{-1}}$, because if the other term is smaller, 
then $\frac{1}{\epsilon} + \frac{1}{\delta}$ is dominated by $1 + \frac{2}{\epsilon}$ anyways. 


Taking $L \in \F_q[T]^+$ with
\begin{equation}
ML \equiv -a \ \mathrm{mod} \ C, \quad \ d(L) \neq k_i - m_i, \quad d(L) \neq d, \quad d(L) \leq c + n + 1,
\end{equation}
and writing $g = L + hC$ with $h \in \mathcal{M}_{d-c}$,
we see that the sum over $c$ contributes
\begin{equation*}
\sum_{c=0}^{d + m - z - 1} (d + m - c) \sum_{C \in \mathcal{M}_c} \sum_{h \in \mathcal{M}_{d-c}} 
\mu\left(\frac{a+ LM}{C} + hM \right)  \prod_{i=1}^n \mu \left(a_i + LM_i  + hCM_i \right)
\end{equation*}
and we again apply \cref{LinearFormsMobThm} to the innermost sum, with the same values of $\epsilon$ and $\delta$.  Everything is the same as before, with $c$ replaced by $b$, except for two things.
\begin{enumerate}

\item We can no longer use the inequality $b \leq \frac{m+d}{2}$, but rather the slightly weaker inequality $c \leq \frac{m+d+1}{2} $ .

\item The term $\mu\left(\frac{a+ LM}{C} + hM \right) $ appears, which means we must check that \begin{equation} (d-c) \geq \epsilon m, (d-c) \geq \epsilon d \left( \frac{a + LM }{ C} \right). \end{equation}

\end{enumerate}

However (1) is no difficulty as we may assume $d$ sufficiently large and perturb the parameters $\epsilon$ and $\delta$ slightly to insure the inequalities of \cref{LinearFormsMobThm} still hold. For this reason we will assume $ c \leq \frac{m+d}{2}$ while handling (2) as well. 

 To check that $(d-c)  \geq \epsilon m $ we observe that, because  $ c \leq \frac{m+d}{2} $ and $ d < \alpha m$, we have  $d-c \geq \frac{ d-m }{2} $, and thus $\frac{d-c}{m} \geq \frac{1-\alpha}{ 2\alpha} \geq \frac{1-\alpha}{1+\alpha} \geq \epsilon  $.

To check that \begin{equation} d \left( \frac{a+ LM }{C} \right) \leq \max( m+d  - c, m +n) \leq \delta^{-1} (d-c) \end{equation} 
we observe 
\begin{equation*} 
\delta \leq  \frac{1-\alpha}{1+ \alpha} = \frac{2 (d-m)}{d+m}=   \frac{ 1}{  \frac{ m}{ (d-m)/2} +1 } \leq  \frac{1}{\frac{m}{d-c} +1 }  = \frac{d-c}{m+d-c}.
\end{equation*}
\end{proof}

\subsection{Singular series}

For nonzero $a \in \F_q[T]$ define
\begin{equation}
\mathfrak{S}_q(a) = \prod_{P \mid a}\left(1 - |P|^{-1} \right)^{-1} \prod_{P \nmid a} \left( 1 - \left(|P|-1 \right)^{-2} \right).
\end{equation}

The following propositon allows us to identify the main term in our twin prime number theorem.
An analogous result over the integers is proved in \cite[Lemma 2.1]{GY03}.

\begin{prop} \label{SingSirProp}

Fix a prime power $q$.
Then for an integer $n \geq 1$ and a nonzero $a \in \F_q[T]$ we have
\begin{equation}
\sum_{k = 1}^n k \sum_{\substack{M \in \mathcal{M}_k \\ (M,a) = 1}} \frac{\mu(M)}{\varphi(M)} = -\mathfrak{S}_q(a) +
(q-1)^{o\left(n +d(a) \right) -n}.
\end{equation}
\end{prop}

\begin{proof} 
We are interested in the sum of the first $n$ coefficients of
\begin{equation}
\begin{split}
Z(u) = u \frac{d}{du} \sum_{k = 1}^\infty u^k \sum_{\substack{M \in \mathcal{M}_k \\ (M,a) = 1}} \frac{\mu(M)}{\varphi(M)} &= 
u \frac{d}{du} \prod_{ P \nmid a}  \left( 1- u^{d(P)} \left(|P|-1 \right)^{-1} \right) \\
&= u \frac{d}{du} \left( (1-u)G(u) \right)
\end{split}
\end{equation}
where $G(u)$ is
\begin{equation*}
\prod_{P \mid a} \left(1- u^{d(P)} |P|^{-1} \right)^{-1}  
\prod_{P \nmid a} \left(1- u^{d(P)} |P|^{-1} \right)^{-1}  
\left( 1- u^{d(P)} \left( |P|-1 \right)^{-1} \right).
\end{equation*}
The sum of all coefficients equals
\begin{equation}
Z(1) = -G(1) = -\mathfrak{S}_q(a)  
\end{equation}
because 
\begin{equation} 
\left(1 - |P|^{-1} \right)^{-1} \left( 1 -  \left( |P|-1 \right)^{-1} \right) = 1 - \left( |P|-1 \right)^{-2}. 
\end{equation}

As in \cref{MainTermPrimesLD}, to prove the bound for the error term it suffices to prove that $G(u)$ is bounded subexponentially in $a$ for $u$ on each circle of radius $<q-1$. 

Note that
\begin{equation}
\begin{split}  
&\left(1- u^{d(P)} |P|^{-1} \right)^{-1}  \left(1- u^{d(P)} \left(|P|-1 \right)^{-1} \right)  = \\
&1  -  \frac{ u^{d(P)} }{ |P| \left(|P|-1 \right)  \left(1- u^{d(P)} |P|^{-1} \right)}
\end{split}
\end{equation} 
so $G(u)$ can be rewritten as  
\begin{equation*}
\prod_P \left( 1  -  \frac{ u^{d(P)} }{ |P| \left(|P|-1 \right)  \left(1- u^{d(P)} |P|^{-1} \right)} \right)
\prod_{P \mid a} \left( 1- u^{d(P)} \left( |P|-1 \right)^{-1} \right)^{-1}.
\end{equation*}
The first product above is independent of $a$ and converges on the disc where $|u|<q$. 
On a circle of radius $r$, the value of each term in the second product is at most 
\begin{equation}
\left(1 -\frac{r^{d(P)}}{q^{d(P)}-1} \right)^{-1} \leq \left( 1- \frac{r}{q-1} \right)^{-1}
\end{equation}
and thus is bounded, and the number of terms is $o(d(a))$, so the product is subexponential in $a$.

\end{proof}

\subsection{Hardy-Littlewood}

\begin{thm}

For every odd prime number $p$, and power $q$ of $p$ with 
\begin{equation}
q> 685090 p^2, 
\end{equation}
there exists $\lambda > 0$ such that the following holds.
For nonnegative integers $d > \ell$, and $a \in \mathcal{M}_\ell$ we have
\begin{equation}
\sum_{f \in \mathcal{M}_d} \Lambda(f)\Lambda(f + a) = \mathfrak{S}_q(a) q^d + O \left( q^{(1 - \lambda)d} \right)
\end{equation}
as $d \to \infty$, with the implied constant depending only on $q$.

\end{thm}

\begin{proof}

Using \cref{ConvId} our sum becomes
\begin{equation}
- \sum_{k=0}^d k \sum_{M \in \mathcal{M}_k} \mu(M) \sum_{N \in \mathcal{M}_{d - k}} \Lambda(a + NM). 
\end{equation}
The appearance of the factor $\mu(M)$ allows us to consider only squarefree $M$,
so by \cref{PrimesLODCor},
the contribution of the range $0 \leq k \leq d/(2 - \omega)$ is
\begin{equation}
-q^d \sum_{k=0}^{d/(2-\omega)} k \sum_{\substack{M \in \mathcal{M}_k \\ (a, M) = 1}} \frac{\mu(M)}{\varphi(M)} + 
O\left(d q^{\frac{d}{2 - \omega}} E_q  \right)
\end{equation}
which equals by \cref{SingSirProp} to
\begin{equation}
q^d\mathfrak{S}_q(a) + 
O\left(d q^{\frac{d}{2 - \omega}} E_q  + q^dE'_q \left( \frac{d}{2 - \omega} \right) \right).
\end{equation}

The error term is of power savings size.

In the other range, we need to prove a power savings bound for 
\begin{equation}
\sum_{k > d/(2 - \omega)}^d k \sum_{N \in \mathcal{M}_{d - k}} \sum_{M \in \mathcal{M}_k} \Lambda(a + MN) \mu(M)  
\end{equation}
and this is done by applying \cref{MobPrimCor} to the innermost sum with 
\begin{equation}
n=1, \quad \tilde{\epsilon}=\infty, \quad \tilde{\delta}=\infty,  \quad \alpha= 1-\omega > \frac{d-k}{k}
\end{equation}
and $\beta>0$ but very small. 
This requires
\begin{equation} q> \left( 2 p  e \left( 1 + \frac{2 +2 -2\omega}{\omega}  \right) \right)^2=  \left( 2 p e \left( \frac{4}{\omega}-1\right) \right)^2 \end{equation}
and \cref{PrimesLODCor} requires
\begin{equation} q> p^2e^2  \left( 1+ \frac{50}{   1- 32 \omega } \right)^2\end{equation}
so the optimal bound is obtained by solving \begin{equation} 2\left(\frac{4}{\omega}-1\right) =  1+ \frac{50}{   1- 32 \omega } \end{equation} whose solution is \begin{equation} \omega= \frac{103- \sqrt{ \frac{30803}{3} } }{64} = .0261\dots \end{equation} which satisfies \begin{equation} \left(  2 e \left( \frac{4}{\omega}-1\right)\right)^2 < 685090.
\end{equation} 
\end{proof}

\section{Results for Small $q$}

We prove \cref{SmallqRes1}.

\begin{proof}

Set $d(f) = d$.
We need to show that by suitably changing the coefficients of $f$ in degree at most $\eta d$,
one can arrive at a polynomial with a given (nonzero) M\"{o}bius value.

Let $c < \eta d$ be the largest even integer not divisible by $3$. 
Note that
\begin{equation} \label{cEstEq}
c \geq \eta d - 4.
\end{equation}
We take the coefficient of $T^c$ in $f$ to be $1$,
and the coefficient of $T^k$ to be $0$ for every $k < c$ that is not divisible by $3$.
Hence, it is enough to show that
\begin{equation} \label{SetSignChangeEq}
1,-1 \in \left \{ \mu(f + b^3) : b \in \mathcal{M}_{\lfloor c/3 \rfloor} \right\}.
\end{equation} 
By \cref{PreparingMobiusToArithProgLem} with $M = 1 , \ a = f$, and $g = b^3$, our set equals
\begin{equation} 
\left \{ S \cdot \chi(w + b^3) : b \in \mathcal{M}_{\lfloor c/3 \rfloor} \right \}
\end{equation}
and since the highest power of $T$ that divides $(f + b^3)' = f'$ is $T^{c-1}$,
we conclude that $S = \pm 1$ and from \cref{PreparingMobiusToArithProgRem} that $\chi$ is a nonprincipal character (to a squarefree modulus $E$). 
Arguing as in \cref{prootexteq} to 'extract third roots', we are thus led to consider
 \begin{equation} 
\left \{ \chi(\widetilde w + b) : b \in \mathcal{M}_{\lfloor c/3 \rfloor} \right \}.
\end{equation}

From \cref{PreparingMobiusToArithProgLem} we further conclude that
\begin{equation} \label{QEstEq}
d(E) \leq d - c + 1 \leq (1 - \eta)d + 5,
\end{equation}
and on the other hand
\begin{equation}
\lfloor c/3 \rfloor \geq \frac{c}{3} - 1 \geq \frac{\eta d- 4}{3} - 1 \geq \frac{\eta}{3}d - 3,
\end{equation}
so combining the two we get
\begin{equation}
\frac{\lfloor c/3 \rfloor}{d(E)} \geq \frac{\frac{\eta}{3}d - 3}{(1 - \eta)d + 5}.
\end{equation}
Since we have assumed that $3/7 < \eta < 1$, the right hand side of the above tends to a quantity greater than $1/4$ as $d \to \infty$.
Consequently, we can use the (function field version of the) Burgess bound (as stated for instance in \cite[Theorem 2]{Bur63}) to show that $1$ and $-1$ belong to the set above. 
Such a version is obtained in \cite{Hsu99}.
\end{proof}

Now we prove \cref{SmallqRes2}.

\begin{proof}

Set $k = d(P)$. For positive integers $d,n$, we seek cancellation in
\begin{equation}
\sum_{g \in \mathcal{M}_d} \mu(a + gP^n)
\end{equation}
where $a \in \F_q[T]^+$ satisfies $d(a) < nk$ and $a \equiv 1 \ \mathrm{mod} \ P^{n-2}$.
We assume first that $3 \mid n$,
and follow the proof of \cref{LinearFormsDistinctDerivatives} up to \cref{AdKanEq} getting
\begin{equation} \label{CopiedEq}
\sum_{d(h) < t} \chi_r(f + h)
\end{equation}
with $\chi_r$ a character mod $E_r$.
If $\chi_r$ is principal, then by \cref{PreparingMobiusToArithProgRem} we have
\begin{equation}
P^{2n}r' + a'P^n = D_r^{(1)} = AB^2, \quad A,B \in \F_q[T], \ A \mid P^n, \ P \nmid B
\end{equation}
and since $P^{n-3} \mid a'$, we conclude that
\begin{equation}
P^{3}r' + \frac{a'}{P^{n-3}} = \widetilde A \widetilde B^2, \quad \widetilde A, \widetilde B \in \F_q[T], \ \widetilde A \mid P.
\end{equation}
There are $\ll q^{d/2}$ choices of $r \in \mathcal{R}$ satisfying the above, so those can be neglected.

For $r \in \mathcal{R}$ with $\chi_r$ nonprincipal, we note that
\begin{equation}
d(E_r) \leq d \left( \mathrm{rad} \left(P^{2n}r' + a'P^n \right) \right) \leq d \left(P^3r' + \frac{a'}{P^{n-3}} \right) + k
\end{equation}
so for large enough $d$ we have $t/d(E_r) > 1/4$,
hence cancellation in \cref{CopiedEq} is guaranteed by Burgess.

Suppose now that $3 \mid n + \beta$ for some $\beta \in \{1,2\}$, and write
\begin{equation}
\sum_{g \in \mathcal{M}_d} \mu(1 + gP^n) = \sum_{d(g_0) < \beta k} \sum_{g_1 \in \mathcal{M}_{d - \beta k}} \mu(1 + g_0P^n + g_1P^{n + \beta})
\end{equation}
for $d \geq \beta$. 
We have thus reduced to the previous case with $a = 1 + g_0P^n$. 
\end{proof}

\begin{remark}
We see from the proof that $1$ is not the only residue class for which the argument works.
Also, the modulus does not have to be a power of a fixed prime,
but it has to be `multiplicatively close' to a cube.
\end{remark}

\appendix

\section{Orthogonality of M\"{o}bius and \\ inverse additive characters}

We explain how a variant of the results of \cite{FM98} carries over to function fields and gives \cref{FKMThm}.

A standard strategy in the treatment of sums such as those from \cref{FMeq} is to use a combinatorial identity for the M\"{o}bius function.
Following \cite{FM98}, we use Vaughan's identity, which for $f \in \F_q[T]^+$ gives
\begin{equation}
\mu(f) = - \mathop{\sum_{d(g) \leq \alpha} \sum_{d(h) \leq \beta}}_{gh \mid f} \mu(g)\mu(h) + 
\mathop{\sum_{d(g) > \alpha} \sum_{d(h) > \beta}}_{gh \mid f} \mu(g)\mu(h)
\end{equation}
where summation is over monic polynomials, 
and $\alpha, \beta$ are nonnegative integers with $\max\{\alpha, \beta\} < d(f)$.
For a proof (that is also valid for function fields) see \cite[Proposition 13.5]{IK04}.

By applying Vaughan's identity as in \cite[Section 6]{FM98},
we reduce our task to bounding sums of type I:
\begin{equation}
\Sigma^{(\text{I})}_{k,r} = \mathop{\sum_{f \in \mathcal{M}_k} \sum_{g \in \mathcal{M}_r}}_{(fg,M) = 1} \gamma_f \psi(\overline{fg})
\end{equation}
where $k \leq \frac{1}{8}d(M) + \frac{7}{16}d, \ k + r \leq d, \ |\gamma_f| \leq 1$ and sums of type II:
\begin{equation}
\Sigma^{(\text{II})}_{k,r} = \mathop{\sum_{f \in \mathcal{M}_k} \sum_{g \in \mathcal{M}_r}}_{(fg,M) = 1} \gamma_f \delta_g \psi(\overline{fg})
\end{equation}
with $k \geq \frac{1}{8}d(M) + \frac{7}{16}d, \ r \geq \frac{7}{16}d - \frac{3}{8}d(M), \ k + r \leq d, \ |\delta_g| \leq 1$.
For every $\epsilon > 0$, we need the bounds
\begin{equation}\label{appendix-required-bound}
\Sigma_{k,r}^{(\text{I})} \ll q^{\left(\frac{3}{16} + \epsilon \right)d(M) + \frac{25}{32}d}, \quad 
\Sigma_{k,r}^{(\text{II})} \ll q^{d + \epsilon d(M) - \frac{r}{2}} + q^{d + \left( \frac{1}{4} + \epsilon \right)d(M) - \frac{k}{2}}
\end{equation}
that are analogous to \cite[Equation 6.4]{FM98}.
The bounds \eqref{appendix-required-bound} then imply \eqref{FMeq} by the argument of \cite[Section 6]{FM98}.

\subsection{Sums  of type I}
Following first the bilinear shifting trick argument of  \cite[\S4]{FM98}, we obtain the inequality 
\begin{equation} \label{FirstBoundTypeOne} 
\Sigma^{(\text{I})}_{k,r} \ll \frac{1}{ V} \sum_{ \substack{ a \in \mathcal M_{d_A} \\ (a,M)=1}} \sum_{ f \in \mathcal M_k}  \sum_{g \in \mathcal M_r}  \left| \sum_{\substack{ b \in \mathcal M_{d_B}\\ ( a^{-1} g + b, M)=1 }} \psi \left( \overline{ a f (a^{-1} g + b) }  \right) \right| 
\end{equation} 
where, in analogy with the variables $A,B$ from \cite[(4.8)]{FM98},
\begin{equation}\label{dadb} 
d_A = \frac{3r - k}{4}, \quad d_B = \frac{k+r}{4} -1 
\end{equation} 
and 
\begin{equation} 
V = \# \left \{ ab : a \in \mathcal M_{d_A} , b \in \mathcal M_{d_B}, (a,M)=1 \right \}  \gg  q^{ d_A + d_B - \epsilon d(M)}.
\end{equation} 
The parameter $t$ and the term $e(-bt)$ from \cite{FM98} do not appear here, 
as they arise from the failure of an archimedean interval  to be perfectly invariant under a shift. 

Now following the H\"{o}lder's inequality argument from \cite[\S4.a]{FM98}, we get as in \cite[Equation 4.6]{FM98}, the bound
\begin{equation} \label{FM46}
\begin{split}
&\left|  \sum_{ \substack{ a \in \mathcal M_{d_A} \\ (a,M)=1}} \sum_{ f \in \mathcal M_k}  \sum_{g \in \mathcal M_r}  \left| \sum_{\substack{ b \in \mathcal M_{d_B}\\ a^{-1} g + b \in A^\times}} \psi \left( \overline{ a f (a^{-1} g + b)} \right) \right| \right|^6  \ll \\
&q^{ 5  ( k+r + d_A) + \epsilon d(m) } \sum_{ b_1,\dots, b_3' \in \mathcal M_{d_B}  }  \left| \sum_{\substack{ h \in A\\ h+b_i \in A^\times \\  h+b_i' \in A^\times } } \sum_{\substack{ d(s) \leq  k+ d_A \\ (s,M)=1 } }  
\psi \left( \Delta_{\mathbf b} (h,s) \right) \right|  
\end{split}
\end{equation}
where $A =  \F_q[T]/(M)$, 
\begin{equation}
\mathbf b = (b_1, b_2, b_3, b_1', b_2', b_3') \in \F_q[T]^{6},
\end{equation}
and the function $\Delta_{\mathbf b} (h,s) $ is defined by
\begin{equation}
\Delta_{\mathbf b}(h,s) = \sum_{i=1}^3 \left[  \frac{ 1} { (h + b_i)s } - \frac{ 1}{  (h + b_i')s} \right].
\end{equation}

Because of our conditions $ h+ b_i \in A^\times ,  h + b_i' \in A^\times ,$ and $(s, M)=1$, 
the function $\Delta_{\mathbf b}$ is never evaluated where it is undefined. 
The H\"{o}lder's inequality argument is slightly more involved because the invertibility assumptions are more complex in case $M$ is not prime, but the idea is essentially the same.

As in \cite{FM98}, we complete the right hand side of \cref{FM46} and get
\begin{equation} \label{CompleteEq}
\begin{split}  
&\sum_{\substack{ h \in A\\  h + b_i \in A^\times \\ h+ b_i' \in A^\times} } \sum_{\substack{ d(s) \leq  k+ d_A \\ (s,M)=1 } }    \psi \left( \Delta_{\mathbf b} (h,s) \right) = \\
&q^{ k + d_A  - d(M)}  \sum_{ \substack { z \in \mathbb F_q[T] \\ d(z) < d(M) - k - d_A }} \sum_{\substack{ h \in A\\ h + b_i \in A^\times \\ h+ b_i' \in A^\times} } \sum_{ s \in A^\times  }   \psi \left( \Delta_{\mathbf b} (h,s) + zs \right).
\end{split}
\end{equation}  
In order to treat the innermost sum on the right hand side, we note that
\begin{equation} 
\Delta_{\mathbf b} (h,s)  +  zs = 
\frac{1}{s}\sum_{i=1}^3 \left[ \frac{1}{ (h+b_i )} - \frac{1}{ (h+b_{i'} )} \right] + zs 
\end{equation}
 so the aforementioned innermost sum over $s$ is a Kloosterman sum.
We put
\begin{equation}
A_{\bf b} = \left\{ x \in A :  (x+ b_1) \cdots (x+ b_3') \in A^\times \right\},
\end{equation} 
define a function $R_{\mathbf b} \colon A_{\mathbf b} \to A$ by 
\begin{equation}
R_{\mathbf b}(x) = \sum_{i=1}^3 \left( \frac{1}{x + b_i} - \frac{1}{x+ b_i'} \right),
\end{equation} 
and a Kloosterman sum
\begin{equation}
S(x,z) = \sum_{y \in A^\times} \psi \left(xy^{-1} + zy \right), \quad x \in A, \ z \in \F_q[T].
\end{equation}
In this notation, for any $z \in \F_q[T]$ we have
\begin{equation} \label{PassToSNot}
\sum_{h \in A_{\bf b} } \sum_{s \in A^\times} \psi \left( \Delta_{\mathbf b} (h,s)  +  zs \right) =
\sum_{x \in A_{\mathbf b}} S(R_{\mathbf b}(x), z) 
\end{equation} 
and the following claim.

\begin{prop} \label{CompleteType1}

For $\sigma \in S_3$ and
\begin{equation}
M_\sigma^{\mathbf b} = \gcd \left( M, b_1 - b_{\sigma(1)}',  b_2 - b_{\sigma(2)}', b_3 - b_{\sigma(3)}' \right)
\end{equation}
we have if $p>3$ the bound
\begin{equation}
\left| \sum_{x \in A_{\mathbf b}} S(R_{\mathbf b}(x), z) \right| \ll 
|M| d_2(M)^4 \left|\operatorname{lcm}_{\sigma \in S_3} \gcd( M_\sigma^{\mathbf b}, z) \right| 
\end{equation}
with the implied constant depending only on $q$. 

If $p=3$, with \begin{equation} M_{\triangle}^{\mathbf b} = \gcd \left ( M , b_1-b_2, b_2-b_3, b_1'-b_2', b_2'-b_3' \right) ,\end{equation} we have the bound \begin{equation}
\left| \sum_{x \in A_{\mathbf b}} S(R_{\mathbf b}(x), z) \right| \ll 
|M| d_2(M)^4 \left|\operatorname{lcm}_{\sigma \in S_3 \cup \{ \triangle \} } \gcd( M_\sigma^{\mathbf b}, z) \right| 
\end{equation}
with the implied constant depending only on $q$. 

\end{prop}

\begin{proof}

Since both the bound and the sum are multiplicative in $M$, it suffices to handle the case when $M$ is prime, where we show that 
\begin{equation}
\left| \sum_{x \in A \setminus \{ b_1,\dots, b_3'\} } S(R_{\mathbf b}(x), m) \right| \leq 16 |A| 
\end{equation}
unless $z=0$ and either
\begin{itemize}

\item for some $\sigma \in S_3$ we have $b_i = b_{\sigma(i)}'$ for all $1 \leq i \leq 3$;

\item or $p = 3$,  $b_1 = b_2 = b_3$, and $b_1' = b_2' = b_3'$;

\end{itemize}
in which case we have the trivial bound
\begin{equation}
\left| \sum_{x \in A \setminus \{ b_1,\dots, b_3'\} } S(R_{\mathbf b}(x), 0) \right| \leq |A|^2. 
\end{equation}

The relevance of these conditions is that the residue of the pole of $R_{\mathbf b}$ at a point $x$ equals 
\begin{equation}
\#\{1 \leq i \leq 3 : b_i = -x\} - \#\{1 \leq i \leq 3 : b_i' = -x\}
\end{equation}
so it is nonzero whenever these two numbers are not equal, except when $p=3$, one of these numbers is $3$, and the other is zero. 
Hence, $R_{\mathbf b}$ has a pole unless each $b_i$ is equal to some $b_{\sigma(i)'}$, or $p=3$, all the $b_i$ are equal, and the $b_i'$ are also all equal. 

Excluding these `trivial' values of $\mathbf b$ for $z= 0$, we get that the rational function $R_{\mathbf b}$ is nonconstant and at most $6$ to $1$.
Hence, if $R_{\mathbf b}(x) \neq 0$ we get $S(R_{\mathbf b}(x), 0) = -1$,
while for the values of $x$ with $R_{\mathbf b}(x) = 0$, at most $6$ in number, we have
\begin{equation}
S(0,0) = |A| - 1.
\end{equation}
In total, we get
\begin{equation}
\left| \sum_{x \in A \setminus \{ b_1,\dots, b_3'\} } S(R_{\mathbf b}(x), 0) \right| \leq |A| + 6|A| = 7|A|.
\end{equation}

Suppose from now on that $z \neq 0$.
Note that if the rational function $R_{\mathbf b}$ is constant then it necessarily vanishes identically, so we get
\begin{equation}
\left| \sum_{x \in A \setminus \{ b_1,\dots, b_3'\} } S(0, z) \right| \leq |A|.
\end{equation}
We can thus assume that $R_{\mathbf b}$ is nonconstant, and let $\mathcal K\ell_2$ be the Kloosterman sheaf defined by Katz.  
With this notation, our sum can be written as 
\begin{equation}
\begin{split}
&- \sum_{x \in A \setminus \{b_1,\dots, b_3'\}} \tr \left(\operatorname{Frob}_{|A|}, (\mathcal K \ell_2)_{ R_{\mathbf b}(x)  z} \right) = \\
&- \sum_{x \in A \setminus \{b_1,\dots, b_3'\}} \tr \left( \operatorname{Frob}_{|A|}, \left( [ z R_{\mathbf b} ]^* \mathcal K \ell_2 \right)_{ x} \right) = \\ 
&- \sum_{i=0}^{2} (-1)^i \tr \left( \Frob_{|A|}, H^i_c \left(\mathbb A^1_{\overline{\mathbb F_q}} \setminus \{- b_1,\dots, -b_3'\},  [z R_{\mathbf b}]^* \mathcal K\ell_2 \right) \right)
\end{split}
\end{equation}
by the Grothendieck-Lefschetz fixed point formula. Because the geometric monodromy group of $\mathcal K\ell_2$ is $SL_2$, which is connected, the geometric monodromy group of its pullback by any finite covering map is $SL_2$, which has no monodromy coinvariants, so the cohomology groups in degree $0$ and $2$ vanish. As $\mathcal K\ell_2$ is pure of weight $1$, its pullback by a finite covering map is mixed of weight at most $1$, so by Deligne's theorem, 
the eigenvalues of $\Frob_{|A|}$ on $H^1_c (\mathbb A^1_{\overline{\mathbb F_q}} \setminus \{- b_1,\dots, -b_3'\},  [z R_{\mathbf b}]^* \mathcal K\ell_2)$ have absolute value at most $|A|$.
Hence, in order to bound our sum by $16|A|$, it suffices to prove that the dimension of the above cohomology group is at most $16$.

By the aforementioned vanishing of cohomology in degrees $0$ and $2$, 
the dimension of our cohomology group equals minus the Euler characteristic. 
Since $ [z R_{\mathbf b}]^* \mathcal K\ell_2$ is lisse of rank $2$ on $\mathbb A^1_{\overline{\mathbb F_q}} \setminus \{- b_1,\dots, -b_3'\}$, its Euler characteristic is twice the Euler characteristic of $\mathbb A^1_{\overline{\mathbb F_q}} \setminus \{ -b_1,\dots, -b_3'\}$, which is 
\begin{equation}
2 \left( 1 - \# \{ -b_1,\dots, -b_3'\} \right),
\end{equation}
minus the sum of the Swan conductors at each singular point. Because the rational function $m R_{\mathbf b}$ has a zero at $\infty$ and a pole of order at most $1$ at each $b_i$ or $b_i'$, 
the Swan conductor of $ [m R_{\mathbf b}]^* \mathcal K\ell_2$ at $\infty$ vanishes and the Swan conductor of $ [z R_{\mathbf b}]^* \mathcal K\ell_2$ at $b_i$ or $b_i'$ is at most $1$, so the total Euler characteristic is at least 
\begin{equation}
2 - 3 \#\{ -b_1, \dots, -b_3' \} \geq 2 - 3 \cdot 6 = -16.
\end{equation}
\end{proof}

\begin{cor} \label{FifthDivCor}

Keeping the same notation, we have if $p>3$
\begin{equation*} 
\sum_{h \in A_{\bf b} } \sum_{\substack{ s \in A^\times \\  d(s) \leq  k + d_A}}  \psi \left( \Delta_{\mathbf b} (h,s) \right)
\ll d_2(M)^5  \left(  |M|+ q^{ \frac{3 (r+k)}{4}}  \left|\operatorname{lcm}_{\sigma \in S_3} M_{\sigma}^{\mathbf b} \right| \right).
\end{equation*} 
and, if $p=3$, the same bound but with $M_{\triangle}^{\mathbf b}$ also included in the $\operatorname{lcm}$.

\end{cor}

\begin{proof}

Using \cref{CompleteEq}, \cref{PassToSNot}, and \cref{CompleteType1} we get a bound of 
\begin{equation}
\ll q^{k + d_A} d_2(M)^4 \sum_{ \substack { z \in \mathbb F_q[T] \\ d(z) < d(M) - k - d_A  }} \left| \lcm_{\sigma \in S_3} \gcd 
\left( M_\sigma^{\mathbf b}, z \right) \right|
\end{equation}
for the left hand side above.
Summing over the possible values of the least common multiple, we get
\begin{equation}
\begin{split}
&\ll q^{\frac{3 (r+k)}{4}} d_2(M)^4 \sum_{L \mid \lcm_{\sigma \in S_3} M_\sigma^{\mathbf b}} |L| \sum_{\substack{z \in \F_q[T] \\ d(z)<   d(M) - \frac{3 (r+k)}{4} \\ L \mid   z}}  1 \\
&\ll q^{\frac{3 (r+k)}{4}} d_2(M)^4 \sum_{{L \mid \lcm_{\sigma \in S_3} M_\sigma^{\mathbf b}}} |L| \max \left\{  q^{  d(M)- \frac{3 (r+k)}{4} - d(L)} , 1 \right\}.
\end{split}
\end{equation}
The contribution of $q^{d(M) - \frac{3 (r+k)}{4} - d(L)}$ (respectively, of $1$) is the first (respectively, the second) summand of the right hand side in our corollary.
\end{proof}

\begin{cor}

Notation unchanged, we have
\begin{equation} \label{ElevBound}
\sum_{ b_1,\dots, b_3' \in \mathcal M_{d_B}  }  
\left| \sum_{h \in A_{\bf b}} \sum_{\substack{ s \in A^\times \\ d(s) \leq  k + d_A} } \psi \left( \Delta_{\mathbf b} (h,s) \right) \right| \ll  |M| d_2(M)^{12}   q^{ \frac{6}{4} ( k+r)}
\end{equation} 

\end{cor}

\begin{proof}

By \cref{FifthDivCor} we have a bound of
\begin{equation}
\ll \sum_{ b_1,\dots, b_3' \in \mathcal M_{d_B}  } d_2(M)^5  \left(  |M|+ q^{ \frac{3 (r+k)}{4}}  \left| \operatorname{lcm}_{\sigma \in S_3} M_{\sigma}^{\mathbf b} \right| \right).
\end{equation}
Summing first over tuples $M_\sigma$ of divisors of $M$ we get
\begin{equation}
\sum_{\sigma \in S_3} \sum_{M_\sigma \mid M} \sum_{ \substack{b_1,\dots, b_3' \in \mathcal M_{d_B}  \\ \forall \sigma \ M_\sigma = M_\sigma^{\mathbf b}}} d_2(M)^5  \left(  |M|+ q^{ \frac{3 (r+k)}{4}}  \left| \operatorname{lcm}_{\sigma \in S_3} M_{\sigma} \right| \right).
\end{equation}
For each such tuple, the conditions
\begin{equation}
M_\sigma = M_\sigma^{\mathbf b}, \quad \sigma \in S_3
\end{equation} 
imply the congruences 
\begin{equation}
b_i \equiv b_{\sigma(i)}' \ \mathrm{mod} \ M_\sigma, \quad 1 \leq i \leq 3, \ \sigma \in S_3
\end{equation}
which determine $b_1', b_2', b_3'$ mod $\lcm_{\sigma \in S_3} M_{\sigma}$ once $b_1, b_2, b_3$ are chosen.
Hence, for each tuple of divisors of $M$,
the number of possible values of $\mathbf b$ is at most
\begin{equation}
\max \left\{ q^{ 6d_B} \left| \lcm_{\sigma \in S_3} M_{\sigma} \right|^{-3}, q^{3d_B} \right\}.
\end{equation}

If $p=3$, we need a slightly more complicated argument. 
There are at most $M_{\triangle}^2$ ways to choose the congruence classes of $\mathbf b$ modulo $M_{\triangle}$, and then choosing $b_1,b_2,b_3$ arbitrarily now determines  $b_1', b_2', b_3' $ mod $\lcm_{\sigma \in S_3 \cup \{ \triangle \} } M_{\sigma}$. 
Because the number of ways to choose $\mathbf b$ modulo $M_{\sigma}$ and then choose $b_1,b_2,b_3$ is 
\begin{equation} 
\begin{cases}  
M_\triangle^2 ( q^{ 3d_B} / M_\triangle^3) \leq q^{ 3 d_B} & \textrm{if } M_\triangle \leq q^{d_B} \\  q^{2 d_B } \leq q^{3 d_B} &\ \textrm {if } M_\triangle > q^{d_B}
\end{cases}
\end{equation} 
in either case the number of possible values of $\mathbf b$ is at most

\begin{equation}
\max \left\{ q^{ 6d_B} \left| \lcm_{\sigma \in S_3 \cup \{\triangle\} } M_{\sigma} \right|^{-3}, q^{3d_B} \right\}.
\end{equation}

Setting $\tau = d(\lcm_{\sigma \in S_3} M_\sigma)$ (or adding $\triangle$ if $p=3$), and taking the maximal possible contribution for every tuple of divisors of $M$, 
we get the bound
\begin{equation}
\max_{0 \leq \tau \leq d(M)} d_2(M)^{|S_3|+1} \left( q^{ 6d_B -3\tau} + q^{3d_B} \right) d_2(M)^5
\left(  |M|+ q^{ \frac{3 (r+k)}{4} + \tau} \right).
\end{equation}
Expanding the brackets above,
we see that each exponent is a linear function of $\tau$, 
hence maximized either at $\tau=0$ or at $\tau = d(M)$.
Using \cref{dadb}, one observes that the maximal terms $q^{ 6d_B + d(M)} $ and $q^{ 3 d_b + \frac{3 (r+k)}{4} + d(M) }$ agree and arrives at the right hand side of \cref{ElevBound}.
\end{proof}

It then follows from \cref{FM46}, \cref{ElevBound}, and the divisor bound that
\begin{equation*} 
\left| \sum_{ \substack{ a \in \mathcal M_{d_A} \\ (a,M)=1}} \sum_{ f \in \mathcal M_k}  \sum_{g \in \mathcal M_r}  \left| 
\sum_{b \in \mathcal M_{d_B}} \psi \left( a f (a^{-1} g + b) \right) \right| \right| \ll q^{ \frac{41}{24} r + \frac{21}{24} k  + ( \frac{1}{6} + \epsilon) d(M) } 
\end{equation*}
and thus (matching the $\ell=3$ case of \cite[(1.2.3)]{FM98} we get
\begin{equation} 
\Sigma^{(\text{I})}_{k,r} \ll q^{ \frac{17}{24} r + \frac{7}{8} k  + ( \frac{1}{6} + \epsilon) d(M) } 
\end{equation}
using \cref{FirstBoundTypeOne}.
Using the fact that $k \leq \frac{1}{8}d(M) + \frac{7}{16}d$ and $k + r \leq d$ we arrive at the first bound in \cref{appendix-required-bound}.

\subsection{Sums of type II}

Keeping the same notation, 
we follow the proof of \cite[Theorem 1.17]{FKM14}.
Applying Cauchy's inequality and Polya-Vinogradov completion as in \cite[Section 3]{FKM14},
we get
\begin{equation} \label{type-II-incomplete} 
\begin{split}
\left| \Sigma^{(\text{II})}_{k,r} \right|^2 &\leq q^k \sum_{g_1,g_2 \in {\mathcal{M}_r}} \overline{ \delta_{g_1}} \delta_{g_2}  \sum_{\substack{ f \in \mathcal{M}_k \\  {(fg_1,M) = (fg_2, M) = 1} }}\psi(\overline{fg_1} - \overline{f g_2} ) \\ 
&\leq q^k  \sum_{g_1,g_2 \in {\mathcal{M}_r}}  
q^{ k - d(M)}  \sum_{ \substack{ h \in \mathbb F_q[T] \\ d(h) < d(M) - k }} \mathcal C( \overline{g_1}- \overline{g_2}, h ). 
\end{split}
\end{equation} 
where
\begin{equation}
\mathcal{C}(g,h) = \sum_{z \in A^\times}   \psi(g \overline{ z})
e_p\left(\mathrm{Tr}^A_{\F_p}(hz) \right), \quad g,h \in \F_q[T].
\end{equation}

We have the following analog of \cite[Proposition 3.1]{FKM14}.

\begin{prop} \label{type-II-complete}

For $g,h \in \F_q[T]$ and $M_{g,h} = \gcd(M,g,h)$ we have
\begin{equation}
\left| \mathcal{C}(g, h) \right| \leq d_2(M) \sqrt{|M|} \sqrt{|M_{g,h} |}.
\end{equation}

\end{prop}

\begin{proof}

Since both $\mathcal{C}(g,h)$ and our putative bound are multiplicative in $M$, 
it suffices to show that for a prime $P$ we have
\begin{equation}
\left| \mathcal{C}(g, h) \right| \leq 2 \sqrt{|P|}
\end{equation}
 unless $g \equiv h \equiv 0 \ \mathrm{mod} \ P$. 
To demonstrate that, take $f \in \F_q[T]/(P)$ with
\begin{equation}
\psi (z) = e_p \left(  \mathrm{Tr}^{\mathbb F_q[T]/(P)}_{\F_p}(fz) \right)
\end{equation}
and note that
\begin{equation}
\mathcal{C}(g,h) =  
\sum_{z \in \left(\F_q[T]/(P)\right)^\times} e_p \left(  \mathrm{Tr}^{\mathbb F_q[T]/(P)}_{\F_p} ( gf z^{-1}  + hz ) \right).
\end{equation}
We have a Weil bound of $2 \sqrt{ |P|}$ for this exponential sum unless the rational function $gfz^{-1}  + hz$ is an Artin-Schreier polynomial, which can only happen if it is constant, as all its poles have order at most $1$.
The latter happens only if $g=h=0$, as desired.
\end{proof}

By \cref{type-II-incomplete} we have
\begin{equation}
\left| \Sigma^{(\text{II})}_{k,r} \right|^2 \leq q^{2k - d(M)} \sum_{L \mid M}  
\sum_{ \substack{ h \in \mathbb F_q[T] \\ d(h) < d(M) - k }} 
\sum_{\substack{g_1,g_2 \in {\mathcal{M}_r} \\ M_{\overline g_1 - \overline g_2,h} = L }}  
\mathcal C( \overline{g_1}- \overline{g_2}, h )
\end{equation}
so from \cref{type-II-complete} and the divisor bound, this is at most
\begin{equation}
q^{2k - \frac{d(M)}{2} + \epsilon d(M)} \sum_{L \mid M} q^{\frac{d(L)}{2}} 
\sum_{ \substack{ h \in \mathbb F_q[T] \\ d(h) < d(M) - k }} 
\sum_{\substack{g_1,g_2 \in {\mathcal{M}_r} \\ M_{g_1 - g_2,h} = L }} 1.
\end{equation}
Since $M_{g_1 - g_2, h} = L$ implies that $g_2 \equiv g_1, \ h \equiv 0 \ \mathrm{mod} \ L$, the above is at most
\begin{equation}
q^{2k - \frac{d(M)}{2} + \epsilon d(M)} \sum_{L \mid M}
q^{\frac{d(L)}{2} + r+ \max ( r - d(L),0) + \max( d(M) -k - d(L),0) }
\end{equation}
so setting $\xi = d(L)$ and applying the divisor bound once again, we arrive at
\begin{equation}
\max_{0 \leq \xi \leq d(M)}
q^{2k + \frac{\xi}{2} + r+ \max \left\{ r - \xi,0 \right\} + \max \left\{ d(M) -k - \xi,0 \right\} - \frac{d(M)}{2} + \epsilon d(M)}.
\end{equation}

As a function of $\xi$, the exponent above is convex, 
so its values do not exceed those at $\xi =0$ and $\xi = d(M)$, which are 
\begin{equation}
q^{ \frac{ d(M)}{2} + \epsilon d(M) + 2r +k }, \quad q^{ r+ 2k + \epsilon d(M) }
\end{equation}
in view of our assumption that $r \leq d \leq d(M)$. 
Taking a square root and using the fact that $k + r \leq d$ we get the second bound in \eqref{appendix-required-bound}.


\begin{remark}
There are (at least) two other potential approaches for bounding our type II sums.
The first is to follow the proof of \cite[Proposition 1.3]{FM98} that uses Bombieri's bound on complete exponential sums \cite[Lemma 4.3]{FM98}.
One then argues as in \cite[Section 5]{FM98}.
The second is to follow the proof of \cite[Theorem 1.4]{FM98} given in \cite[Section 7]{FM98}.

\end{remark}

%
%

\end{document}